\documentclass[a4paper,11pt]{scrartcl}

\pdfoutput=1

\usepackage[utf8]{inputenc}
\usepackage[T1]{fontenc}
\usepackage{lmodern}
\usepackage[english]{babel}
\usepackage{amsthm, amssymb, amsmath, amsfonts, mathrsfs, dsfont, esint,textcomp}
\usepackage[colorlinks=true, pdfstartview=FitV, linkcolor=blue, citecolor=blue, urlcolor=blue,pagebackref=false]{hyperref}
\usepackage{mathtools}
\usepackage{enumitem}

\usepackage[backend=bibtex,firstinits,maxnames=4,style=alphabetic,isbn=false, date=year, doi=false, eprint= true, url= false]{biblatex}

\usepackage{a4wide}  

\usepackage{microtype}

\usepackage[notcite,notref,color,final]{showkeys}
\definecolor{labelkey}{gray}{.8}
\definecolor{refkey}{gray}{.8}

\definecolor{darkblue}{rgb}{0,0,0.7} 
\definecolor{darkred}{rgb}{0.9,0.1,0.1}
\definecolor{darkgreen}{rgb}{0,0.5,0}

\newtheorem{thm}{Theorem}[section]
\newtheorem{prop}[thm]{Proposition}
\newtheorem{lem}[thm]{Lemma}

\theoremstyle{remark}
\newtheorem{rem}[thm]{Remark}
\theoremstyle{definition}
\newtheorem{defi}[thm]{Definition}

\newcommand\norm[1]{\left\lVert#1\right\rVert}
\newcommand\set[1]{\left\{#1\right\}}
\newcommand\bra[1]{\left({#1}\right)}

\newcommand\abs[1]{\left\lvert#1\right\rvert}

\renewcommand{\le}{\leqslant}
\renewcommand{\ge}{\geqslant}
\renewcommand{\leq}{\leqslant}
\renewcommand{\geq}{\geqslant}
\newcommand{\ls}{\lesssim}
\newcommand{\gs}{\gtrsim}
\renewcommand{\subset}{\subseteq}

\newcommand{\cH}{\mathcal{H}}
\newcommand{\J}{\mathcal{J}}
\newcommand{\fP}{\mathcal{P}}
\newcommand{\mP}{\mathcal{P}}

\newcommand{\N}{\mathbb{N}}

\newcommand{\1}{\mathbf{1}}
\newcommand{\R}{\mathbb{R}}

\renewcommand{\P}{\mathbb{P}}

\newcommand{\red}{\color{red}}

\newcommand{\dd}{\, \mathrm{d}}

\newcommand{\eff}{\mathrm{eff}}

\newcommand{\clos}[1]{\overline{#1}}

\newcommand{\loc}{\mathrm{loc}}
\newcommand{\app}{\mathrm{app}}
\newcommand{\W}{\mathcal{W}}
\newcommand{\dmin}{d_{\mathrm{min}}}

\DeclareMathOperator*{\esssup}{esssup}
\DeclareMathOperator{\curl}{curl}

\DeclareMathOperator{\Id}{Id}

\DeclareMathOperator*{\dv}{div}

\DeclareMathOperator*{\supp}{supp}

\DeclareMathOperator{\dist}{dist}

\numberwithin{equation}{section}

\mathtoolsset{showonlyrefs}

\bibliography{effviscosity.bib}

\begin{document}

\title{The Influence of Einstein's Effective Viscosity on Sedimentation at Very Small Particle Volume Fraction}
\author{Richard M. H\"ofer, Richard Schubert}

\maketitle

\begin{abstract}
We investigate the sedimentation of identical inertialess spherical particles  in a Stokes fluid
in the limit of many small particles. It is known that the presence of the particles leads to an increase of the effective viscosity of the suspension. By Einstein's formula this effect is of the order of the particle volume fraction $\phi$. The disturbance of the fluid flow responsible for this increase of viscosity is very singular (like $\abs{x}^{-2}$). Nevertheless, for well-prepared initial configurations and $\phi\to 0$, we show that the microscopic dynamics is approximated to order $\phi^2 |\log \phi|$ by a macroscopic coupled transport-Stokes system with an effective viscosity according to Einstein's formula.  We provide quantitative estimates both for convergence of the densities in the $p$-Wasserstein distance for all $p$ and for the fluid velocity in Lebesgue
spaces in terms of the $p$-Wasserstein distance of the initial data.
Our proof is based on approximations through the method of reflections and on a generalization of a classical result on convergence to mean-field limits in the infinite Wasserstein metric by Hauray.
\end{abstract}

\section{Introduction}

Let $N \in \N$ and consider an initial distribution of particles $B_i=B_{R}(X_i), i=1,\dots,N$ of radius $R$, sedimenting in a fluid.
We assume that the  domain 
\begin{align}
	\Omega_N := \R^3 \setminus \bigcup_{i=1}^{N} \clos{B_i}
\end{align}
is occupied by a fluid which is modeled by the following Stokes equations in dimensionless form: 
\begin{align}\label{eq:u_N}
\left\{\begin{array}{rcll}
	- \Delta u_N + \nabla p &=&  0, \quad \dv u_N = 0  &\text{in } \Omega_N, \\
	eu_N &=&0 &\text{in } \clos{B_i}, ~ 1 \leq i \leq N, \\
	\int_{\partial B_i}\sigma[u_N]n&=&\frac 1 N g &\text {for all }i=1,\dots,N,\\
	\int_{\partial B_i}(x-X_i)\wedge(\sigma[u_N]n)&=&0 &\text{for all }i=1,\dots,N.\\
\end{array}\right.
\end{align}
Here, $n$ is the inner normal vector at the balls $B_i$ (the outer normal of $\Omega_N$), and the stress tensor $\sigma[u_N]$ is given by
\begin{align}
\sigma[u_N]=2eu_N-p\Id,\qquad \text{where }\quad eu_N=\frac 1 2\bra{\nabla u_N+(\nabla u_N)^T}.
\end{align}
Moreover, $g$ is a constant vector accounting for the gravity.
The pressure $p$ can be viewed as a Lagrange multiplier corresponding to the incompressibility. We do not provide convergence results for the pressure and will therefore denote all the appearing pressure by $p$.

The boundary conditions in \eqref{eq:u_N} are known as sedimentation boundary conditions. The last two lines reflect that the particles are inertialess, and $e u_N = 0$ is imposed since the particles are assumed to be rigid.
Indeed, $e u_N = 0$ in $\overline{B_i}$ is equivalent to the existence of $V_i$ and $\omega_i$ such that $u_N(x) = V_i + (x- X_i) \times \omega_i$ for all $x \in \overline{B_i}$.

The problem becomes dynamic by complementing \eqref{eq:u_N} with the equation of motion for the particle centers:
\begin{align} \label{eq:def.V_i}
\dot{X_i} = V_i :=u_N(X_i),
\end{align}
which renders the domain $\Omega_N$ time dependent.
For details regarding the physical assumptions of this model and related questions, we refer the reader to the introduction of \cite{Hofer18MeanField}.

Our aim is to accurately describe the macroscopic behavior of this system in the limit of many small particles $N \to \infty$, $R \to 0$. ($X_i$, $B_i$ and $R$ implicitly depend on $N$). 
We thus study the dynamics of the empirical distribution of the particles 
\begin{align} \label{eq:rho_N}
\rho_N(t)=\frac 1 N \sum_{i=1}^N \delta_{X_i(t)},
\end{align}
that satisfies the transport equation
\begin{align}
\partial_t \rho_N+u_N\cdot \nabla \rho_N=0.
\end{align}
Assuming $\rho_N(0) \to \rho_0$ in some appropriate sense,
we will show that $\rho_N(t)$ is well described by $\rho(t)$ up to an error of order $
\phi_N^2 |\log \phi_N|$,
 where $\rho$ solves a coupled transport-Stokes system with an effective Einstein viscosity $\mu_\eff = 1 + 5/2 \phi_N \rho$ and where the particle volume fraction $\phi_N$ is defined as
\begin{align} \label{def:phi}
 \phi_N = \frac{4 \pi}{3} N R^3.
\end{align}

\subsection{Previous results}

A first result regarding the sedimentation of spherical particles in Stokes flows has been obtained by Jabin and Otto in \cite{JabinOtto04}. They identified the regime that is sufficiently dilute (corresponding to $NR  \ll 0$ in the setting above) such that the direct effect of the gravity on each particle dominates over the particle interaction through the fluid. Each particle then settles approximately as if it was alone in the fluid.

The first author showed in \cite{Hofer18MeanField} that for well-prepared initial configurations, $\rho_N(t) \to \tilde\tau(t)$ for all $t \geq 0$, where $\tilde \tau$ solves the coupled transport-Stokes system
\begin{align}  \label{eq:v.tau.gammastar}
\left\{\begin{array}{rl}
-\Delta \tilde v+\nabla p&=\tilde \tau g,\\
\dv v&=0,\\
\partial_t \tilde  \tau +(\tilde v + \bra{6 \pi \gamma_\ast}^{-1} g) \cdot \nabla \tilde \tau&=0,\\
\tilde \tau(0)&=\rho_0.
\end{array}\right.
\end{align}
Here, $\gamma_\ast := \lim_{N \to \infty} NR $ quantifies the interaction strength between the particles. The result in \cite{Hofer18MeanField}
is proven for any $\gamma_\ast \in (0, \infty]$. In \cite{Hofer18MeanField} the conditions for the initial configurations consist in a convergence assumption for the empirical density, the conditions that the particles
are well-separated in the sense of
\begin{align} \label{eq:density.condition} \tag{H1}
	\exists c > 0 \, \forall N \in \N \quad  \frac 1 c N^{-1/3} \geq d_{\min}(0) \coloneqq \min_{i\neq j}\abs{X_i(0)-X_j(0)}\ge cN^{-1/3},
\end{align} 
and the condition that the particle volume fraction  is sufficiently small in the sense of
\begin{align} \label{eq:ass.phi.logN} \tag{H2}
 	 \lim_{N \to \infty} \phi_N \log N  = 0.
\end{align}
Note that the first inequality in \eqref{eq:density.condition} always holds if $\rho_N$ converges to some macroscopic density.

Mecherbet showed in \cite{Mecherbet19} that taking into account particle rotations, which are neglected in \cite{JabinOtto04} and \cite{Hofer18MeanField}, does not affect the convergence result. The results in \cite{Mecherbet19}, which hold under different assumptions than those in \cite{Hofer18MeanField} (less restrictive separation condition on the one hand but sufficiently small interaction on the other hand) also contain quantitative estimates for the convergence $\rho_N \to \tilde \tau$ in Wasserstein metrics.

\medskip

On the other hand, it is known that the particles change the effective 
viscosity of the fluid. According to Einstein's formula \cite{Ein06}, this effective viscosity is  given by $\mu_\eff = 1 + 5/2 \phi$, to first order in the (local) particle volume fraction $\phi$.

Rigorous mathematical results have only been obtained in recent years. Haines and Mazzucato  \cite{HM12} proved Einstein's formula for periodic particle configurations on the level of the dissipation rate under straining motion. 

The first result on the level of convergence of the fluid velocity is due to Niethammer and the second author \cite{NiethammerSchubert19}. In \cite{NiethammerSchubert19}, a similar system to \eqref{eq:u_N} is considered.  Instead of the gravity $g/N $ that acts through the particles on the fluid,
 a right-hand side $f$ is introduced in the Stokes equations.
Under  assumptions \eqref{eq:density.condition} and \eqref{eq:ass.phi.logN} as well as convergence of the empirical density $\rho_N$, it is  shown that $\|u_N - u\|_{L^p} = o(\phi_N)$,
where $u$ solves
\begin{equation} 
\begin{aligned}
\left\{\begin{array}{rl}
-\dv((2+5\phi_N \rho_0) eu)+\nabla p&=(1 - \phi_N) f,\\
\dv u&=0.
\end{array}\right.
\end{aligned}
\end{equation}
This result has been generalized to polydispersed particles of more general shape by Hillairet and Wu in \cite{HillairetWu19}, where they also removed any condition $\phi_N \to 0$.
G\'erard-Varet \cite{Gerard-Varet19} and G\'erard-Varet  and the first author \cite{Gerard-VaretHoefer20} were able to considerably relax the separation condition \eqref{eq:density.condition} allowing to treat a large class of random particle configurations.

For results regarding the analysis of the higher order correction in $\phi$, we refer to the paper by G\'erard-Varet and Hillairet \cite{Gerard-VaretHillairet19}, G\'erard-Varet and Mecherbet \cite{Gerard-VaretMecherbet20} and Duerinckx and Gloria \cite{DuerinckxGloria19,DuerinckxGloria20}.

All the mentioned results on the effective viscosity concern the quasi-static case where only the system \eqref{eq:u_N} is studied without coupling it to the dynamical evolution of the particle positions.

\medskip

\subsection{Main results}
The main result of this paper concerns the coupling of  Einstein's formula for the effective viscosity to the dynamical problem of particle sedimentation. To our knowledge, this is the first rigorous result on the validity of Einstein's formula in the dynamical case.

Our main results are formulated in terms of the $p$-Wasserstein metric $\W_p(\cdot,\cdot)$, $1 \leq p \leq \infty$. For the definition we refer to Appendix \ref{app:Wasserstein}. Classical results on Wasserstein distances can be found for instance in \cite{Santambrogio15}.


Our first result shows that the basic transport-Stokes system \eqref{eq:v.tau.gammastar} approximates the microscopic system well up to an error of order $\phi_N$.   More precisely, to ensure that the leading order error is not due to the difference of $NR$ and $\lim_{N \to \infty} NR$, we consider the more accurate system
\begin{align}  \label{eq:v.tau}
\left\{\begin{array}{rl}
-\Delta v+\nabla p&=\tau g,\\
\dv v&=0,\\
\partial_t \tau +(v + \bra{6 \pi \gamma_N}^{-1} g) \cdot \nabla \tau&=0,\\
\tau(0)&=\rho_0,
\end{array}\right.
\end{align}
where $\gamma_N = NR$. Note that $v,\tau$ implicitly depend on $N$. Also note that for a specific $N$ the constant velocity $(6 \pi \gamma_N)^{-1} g$ could be absorbed by considering a coordinate system that moves with this constant speed.

\begin{thm} \label{th:tau}
Assume that assumptions \eqref{eq:density.condition}, \eqref{eq:ass.phi.logN} hold and let $\rho_0 \in L^\infty(\R^3) \cap \mP(\R^3)$, where $\mP(\R^3)$ denotes the space of probability densities. 
Then, for all $T_\ast>0$ and all $N = N(T_\ast)$ sufficiently large there exists a constant $C$ such that 
\begin{align} \label{eq:min.dist.conserved}
	\dmin(t) \geq \dmin(0) e^{-Ct} \quad \text{for all } t \leq T_\ast,
\end{align}
with $C = C(T_\ast,c)$, where $c$ is the constant from \eqref{eq:separation.condition}.
In particular, the solution to \eqref{eq:u_N} -- \eqref{eq:def.V_i} exists on $(0,T_\ast)$.
Moreover, for all $1 \leq  p \leq \infty$, and all $t \leq T_\ast$
\begin{align} \label{eq:rho_N.tau}
  \W_p (\rho_N(t),\tau(t)) +  \leq C ( \phi_N + \W_p(\rho_N(0), \rho_0)) e^{C t},
\end{align}
where $\tau$ is the unique solution to \eqref{eq:v.tau} and  $C=C(T_\ast,c,\norm{\rho_0}_{L^\infty})$.

Furthermore, for all $q < 3$ and all $p> \max\{1,\frac {3q}{3+q} \}$ , there exists $C=C(T_\ast,c,\norm{\rho_0}_{L^\infty},p,q)$, such that for all $t \leq T_\ast$
\begin{align} \label{eq:tau.v-u_N}
\norm{v(t)-u_N(t)}_{L^q_\loc(\R^3)} \leq C ( \phi_N + \W_p(\rho_N(0), \rho_0)) e^{C t}.
\end{align}
\end{thm}

Theorem \ref{th:tau} is a quantitative version of the convergence result obtained by the first author in \cite{Hofer18MeanField}.
A similar quantitative result has already been proven in \cite{Mecherbet19}. We emphasize, though, that we cannot just apply the results in \cite{Mecherbet19} since they require $NR$ to be sufficiently small in our setting. In fact, as we will discuss below in more detail,
we are mainly interested in the case $NR \to \infty$; otherwise the discretization error turns out to be always larger than $\phi$ and thus dominant over the effect of the increase of the viscosity.

Note that we imagine our continuous densities to be limits of empirical distributions and we therefore assume all appearing densities to be probability distributions. All our results though, can, without effort, be generalized to distributions of general (but finite) mass. 
\medskip

We now state the main result of the paper, that compares $\rho_N$ with the solution $\rho_\eff$ of the effective macroscopic system \eqref{eq:rho_eff.u_eff}.

\begin{thm} \label{th:main}
Assume that
\eqref{eq:density.condition}, \eqref{eq:ass.phi.logN} hold and let $\rho_0 \in W^{1,1}(\R^3) \cap W^{1,\infty}(\R^3) \cap \mP(\R^3)$.
Then, for all $T_\ast>0$,  all $N = N(T_\ast)$ sufficiently large, for all $ 1 \leq p < \infty$, and all $t \leq T_\ast$, 
\begin{align} \label{eq:rho_N.rho_eff}
\begin{aligned}
  &\W_p (\rho_N(t),\rho_\eff(t))\\
  &\leq C \left( \phi_N^2 |\log \phi_N| +  \phi_N \W_\infty(\rho_N(0),\rho_0) |\log \W_\infty(\rho_N(0),\rho_0)| + \W_p(\rho_N(0), \rho_0) \right) e^{C t},
\end{aligned}
\end{align}
where  $C=C(T_\ast,c,p,\norm{\rho_0}_{W^{1,1}\cap W^{1,\infty}})$
and  $u_\eff, \rho_\eff$ is the unique solution to
\begin{align} \label{eq:rho_eff.u_eff}
\left\{\begin{array}{rl}
-\dv((2+5\phi_N \rho_\eff)e u_\eff)+\nabla p&=\rho_\eff g,\\
\dv u_\eff&=0,\\
\partial_t \rho_\eff+\bra{u_\eff+(6\pi \gamma_N)^{-1}g}\cdot \nabla \rho_\eff&=0, \\
\rho_\eff(0) &= \rho_0,
\end{array}\right.
\end{align}
where $\gamma_N=RN$.

Furthermore,  there exists a constant  $C=C(T_\ast,c,\norm{\rho_0}_{W^{1,1}\cap W^{1,\infty}},p,q)$, such that for all $q < 3$, all $p>  \max\{1,\frac {3q}{3+q} \}$ and all $t \leq T_\ast$
\begin{align} \label{eq:main.u_eff-u_N}
&\norm{u_\eff(t)-u_N(t)}_{L^q_\loc(\R^3)} \\
  &\leq C \left( \phi_N^2 |\log \phi_N| +  \phi_N \W_\infty(\rho_N(0),\rho_0) |\log \W_\infty(\rho_N(0),\rho_0)| + \W_p(\rho_N(0), \rho_0) \right) e^{C t}. 
\end{align}
\end{thm}

\smallskip

Several comments are in order.

First, we emphasize the perturbative nature of our result. Imposing assumption \eqref{eq:ass.phi.logN}, we can only treat the case when the  particle volume fraction $\phi_N$ vanishes sufficiently fast. The disturbance of the fluid flow that is responsible for the increase in viscosity is very singular (like $\abs{x}^{-2}$) thus we are unable to control the interparticle distance for finite $\phi_N$ even for very short times. We will further comment on this limitation in Section \ref{sec:MoR}. We overcome this problem for $\phi_N\to 0$ by controlling sums $R^3 \abs{X_i - X_j}^{-3}$
due to \eqref{eq:ass.phi.logN} and by exploiting the convergence from Theorem \ref{th:tau}. Note that the relevant interaction for Theorem \ref{th:tau} is less singular (like $\abs{x}^{-1}$). 

The singular nature of the interaction is also responsible for the appearance of the logarithmic correction term in the error $\phi^2 |\log \phi|$ in \eqref{eq:rho_N.rho_eff}.
According to the results on the second order corrections of the effective viscosity
\cite{Gerard-VaretHillairet19,Gerard-VaretMecherbet20} one should expect that the optimal error is $\phi^2$. We remark that the other error term $\phi_N \W_\infty(\rho_N(0),\rho_0) |\log \W_\infty(\rho_N(0),\rho_0)|$ can be absorbed by Young's inequality into the other two terms on the right-hand side of \eqref{eq:rho_N.rho_eff} if the initial data are well-prepared, e.g., if $W_\infty(\rho_N(0),\rho_0) \lesssim W_p(\rho_N(0),\rho_0) \to 0$

\smallskip

Next, we comment on the fact that the discretization in general imposes a constraint on the rate of convergence of the initial distributions $\rho_N \to \rho_0$. Indeed, let $\Gamma(\rho_0,\rho_N(0))$ be the set of all couplings between $\rho_0,\rho_N(0)$ (see Appendix or \cite{Santambrogio15}). Then
\begin{align} \label{eq:discretization.error}
\begin{aligned}
	\W_p(\rho_N(0),\rho_0) 
	&= \inf_{\gamma \in \Gamma(\rho_0,\rho_N(0))} \left(\int_{\R^3 \times \R^3} |x-y|^p \dd \gamma(x,y)\right)^{1/p} \\
	&\geq c N^{-1/3} \left(\int_{\R^3\setminus \cup_i B_{c N^{-1/3}}(X_i)} \rho_0(x) \dd x \right)^{1/p} \geq c' N^{-1/3}
	\end{aligned}
\end{align}
for $c,c'$ sufficiently small depending only on $\|\rho_0\|_\infty$.
Thus, in order that the discretization error is not dominant over the effect of the Einstein correction,
a necessary condition is $N^{-1/3} \leq \phi_N$.
We emphasize that our assumptions allow this condition to be satisfied since \eqref{eq:ass.phi.logN} allows $\phi_N$ to vanish very slowly. 
Note that $N^{-1/3} \leq \phi_N$ in particular implies
$\gamma_N^{-1} \sim (N R)^{-1} = \phi_N^{-1/3} N^{-2/3} \leq N^{- 5 /9}$. Thus, the self-interaction term always vanishes in the limit $N \to \infty$ in this case.


We also remark that it is possible to slightly refine the estimates in  \eqref{eq:rho_N.tau} and \eqref{eq:rho_N.rho_eff}. In particular, it is possible to derive an estimate of the form 
\begin{align}
	W_p(\rho_N(t),\tau(t)) \leq W_p(\rho_N(0),\rho_0) e^{Ct} + (C \phi_N + O(N^{-1/3}))(e^{Ct} - 1),
\end{align}
and similarly for $W_p(\rho_N(t),\rho_\eff(t))$, such that equality holds if one evaluates at $t = 0$. For the sake of simplicity of the presentation, we refrain from providing these estimates. Instead, relying on \eqref{eq:discretization.error} and assumption \eqref{eq:density.condition}, we choose to absorb the errors $O(N^{-1/3})$ in the statement of our results. 

Furthermore, it is also possible, to give estimates for the fluid velocities in \eqref{eq:tau.v-u_N} and \eqref{eq:main.u_eff-u_N} for $q > 3$. The reason why we choose to give the statement for $q < 3$ only is due to  the error coming from the self-interaction.
For $q < 3$, this error can be absorbed because of the integrability of $1/|x|^q$.
On the other hand, for $q > 3$, one would need to add an error $\phi_N^{1/q} (NR)^{-1}$ on the right-hand side.

\smallskip

Finally, we emphasize that to first order in $\phi_N$, the (time-dependent) effective viscosity is  fully captured by the particle density $\tau$ that solves \eqref{eq:v.tau}. 
Indeed, as Theorems \ref{th:tau} and \ref{th:main} suggest,
$(\tau - \rho_\eff) \sim \phi_N$ and thus 
\begin{align}
	\mu_\eff := 1 + 5/2 \phi_N \rho_\eff  =  1 + 5/2 \phi_N \tau + O(\phi_N^2).
\end{align}
In fact, for proving Theorem \ref{th:main}, we will introduce the following intermediate model, which approximates \eqref{eq:rho_eff.u_eff} up to errors  of order $\phi_N^2$:
\begin{equation} \label{eq:u.rho}
\begin{aligned}
\left\{\begin{array}{rl}
-\dv(2 e u +5\phi_N \tau ev)+\nabla p&=\rho g,\\
\dv u&=0,\\
\partial_t \rho+\bra{u+(6\pi \gamma_N)^{-1}g}\cdot \nabla \rho&=0, \\
\rho(0) &= \rho_0.
\end{array}\right.
\end{aligned}
\end{equation}
We emphasize that the restriction $p<\infty$ in Theorem \ref{th:main} does not come from a lack of control for the transition from the microscopic to the macroscopic model but from the lack of an $L^\infty$-theory for the Stokes equation which is needed for comparing the intermediate problem \eqref{eq:u.rho} with the effective model \eqref{eq:rho_eff.u_eff}. 

\subsection{Outline of the paper}

The basic strategy of the proof of both Theorem \ref{th:tau} and Theorem \ref{th:main} is almost the same and mainly consists in two steps.
First, we show, that the microscopic dynamics \eqref{eq:u_N}--\eqref{eq:def.V_i} can be approximated up to an error of order $\phi_N$, respectively $\abs{\log \phi_N}\phi_N^2$, by an explicit system of two-particle interactions. Indeed, one difficulty of the analysis of the microscopic system is that the particle velocities $V_i$ are only given implicitly through the solution of the boundary value problem \eqref{eq:u_N}.
Based on such an explicit approximation, we then prove in a second step the convergence to its mean-field limit in Wasserstein metrics. 

For the proof of Theorem \ref{th:tau}, the approximation 
that we obtain in the first step reads
\begin{align} \label{eq:V_i.strategy1}
	V_i = \frac{g}{6 \pi N R} +  \sum_{j \neq i} \Phi(X_i - X_j)g +  E_i.
\end{align}
Here $\Phi$ is the fundamental solution of the Stokes equations,
\begin{align} \label{eq:Phi}
	\Phi(x) = \frac{1}{8\pi} \bra{\frac{\Id}{\abs{x}}+\frac{x\otimes x}{\abs{x}^3}},
\end{align}
 and $E_i$ is an error term which is of order $\phi_N$
as long as we have good control over the particle configuration.


In the case of the proof of Theorem \ref{th:main}, the approximation for the particle velocities has to be refined.
Here, we show instead of \eqref{eq:V_i.strategy1} that
\begin{align} \label{eq:V_i.strategy2}
	V_i = \frac{g}{6 \pi N R}  + \sum_{j \neq i} \Phi(X_i - X_j)g - 5 \phi_N (e \Phi \ast (\tau (e \Phi g \ast \tau)))(X_i) +  \bar E_i,
\end{align}
where $\bar E_i$ is an error term which is of the order $O(\phi_N^2 | \log \phi_N|)$. Here, we exploit that we already know from Theorem \ref{th:tau} that $\rho_N$ is well approximated by $\tau$ up to an error of order $\phi_N$.
The (formal) mean-field limit of \eqref{eq:V_i.strategy2} is given by system \eqref{eq:u.rho}.
Thus, in a final step of the proof of Theorem \ref{th:main}, we need to show that the solution to \eqref{eq:u.rho} is close to the solution of  system \eqref{eq:rho_eff.u_eff}. 

\medskip

The rest of the paper is organized as follows.
In Section \ref{sec:Strategy}, we explain in more detail the outline of the proof of the main results. In Section \ref{sec:MoR}, we review results based on the so-called method of reflections which enable us to obtain good approximations for the particle velocities. 
In Section \ref{sec:meanfield}, 
we state an abstract result (Theorem \ref{th:mean.field.general}), which will enable us to obtain the desired estimates in the Wasserstein metric for $(\rho_N,\tau)$ and $(\rho_N,\rho)$. This theorem is a generalization of a classical result on mean-field limits by Hauray \cite{Hauray09}. We are able to relax the assumption on the particle initial configuration and to include error terms $E_i$ as above. We therefore believe that this result might be of independent interest.
Finally, in Section \ref{sec:transition}, we outline how to estimate the difference between solutions to \eqref{eq:u.rho} and \eqref{eq:rho_eff.u_eff}.
In this section, we also state necessary well-posedness and regularity results for systems \eqref{eq:v.tau}, \eqref{eq:u.rho} and \eqref{eq:rho_eff.u_eff}.

In Section \ref{sec:meanfieldproof}, we prove the abstract convergence result, Theorem \ref{th:mean.field.general}.

Section \ref{sec:proof.main} contains the proof of the main results  Theorems \ref{th:tau} and \ref{th:main}.

Finally, in Section \ref{sec:continuous}, we give the proofs of the results for the macroscopic equations stated in Section \ref{sec:transition}.  For this, we rely on stability estimates for the systems  \eqref{eq:v.tau}, \eqref{eq:u.rho} and \eqref{eq:rho_eff.u_eff} in Wasserstein spaces and on DiPerna-Lions theory for transport equations \cite{DiPernaLions89}.


\section{Outline of the proof of the main results} \label{sec:Strategy}

\subsection{Explicit approximations of the particle interactions by the method of reflections}\label{sec:MoR}

For the approximation of the microscopic dynamics by explicit two-particle interactions, we rely on the so-called method of reflections.
Variants of this method have been used for related problems, 
notably homogenization problems of Poisson and Stokes equations,
in particular results on the effective viscosity as well as problems of particle sedimentation in \cite{FigariOrlandiTeta85, Rubinstein1986, JabinOtto04,HoferVelazquez18, Hofer18MeanField, Mecherbet19, NiethammerSchubert19,HillairetWu19,Gerard-VaretHillairet19,Hoefer19}.
We will here only give a brief introduction of the method and state the necessary results that we will apply.
For more details, we refer to \cite{Hoefer19} and the references therein.

The method of reflections yields a series expansion for the solution of boundary value problems such as the solution $u_N$ to \eqref{eq:u_N} in terms of the solution operators of single particle problems. The series is convergent 
for sufficiently dilute particle configurations. In our setting, such diluteness is provided by assumptions \eqref{eq:density.condition} and \eqref{eq:ass.phi.logN}.
Note that these two conditions imply  
	\begin{align} \label{eq:c_0}
		c_0 := \frac{R^3}{\dmin(0)^3} \to 0.
	\end{align}
In particular, for $N$ sufficiently large, the particles cannot overlap and 
\begin{align} \label{eq:separation.condition}
	B_{2 R}(X_i(0)) \cap B_{2 R}(X_j(0)) = \emptyset \qquad \text{for  all } i \neq j.
\end{align} 

To apply the method of reflections, we begin by defining the solution to the problem where only one particle is present,
\begin{align} \label{eq:w_N}
	\left\{ \begin{array}{rcl}
		- \Delta w_N + \nabla p &=& \frac{g}{N} \delta_{\partial B_R(0)},  \\
		\dv w_N &=& 0,
	\end{array} \right.
\end{align}
where $\delta_{\partial B_R(0)}=\abs{\partial B_R(0)}^{-1}\cH^2_{|\partial B_R(0)}$. Note that $w_N$ is defined in such a way that the stress condition $\int_{\partial B_R(0)}\sigma[w_N]n=\frac g N$ is satisfied.
Then, as a zero-order approximation for $u_N$, we take the sum of the one-particle solutions
\begin{align}
	v^{(0)}_N(x) := \sum_i w_N(x - X_i).
\end{align}
We observe that $v^{(0)}_N$ satisfies \eqref{eq:u_N} except for the second line, the constraint that the velocity field is a rigid body motion at the particles.

The method of reflections now consists in adding corrections to $v^{(0)}$ in order to fulfill this constraint. To this end, one defines the correction needed for particle $i$ by the operator $Q_i$
through the problem
\begin{align} 
	\left\{\begin{array}{rcll}
		- \Delta Q_i \varphi + \nabla p &=& 0, \quad \dv Q_i \varphi = 0 \quad &\text{in } \R^3 \setminus \overline{B_i}, \\
		e Q_i \varphi &=& e  \varphi \quad &\text{in } \overline{B_i}.
	\end{array} \right.
\end{align}
Then, the $k$-th order approximation through the method of reflections reads
\begin{align} \label{eq:def.v^k}
	v^{(k)}_N := (1- \sum_i Q_i)^k v^{(0)}_N.
\end{align}

We will rely on the convergence result from \cite{Hoefer19} 
under a smallness condition of $c_0$ from \eqref{eq:c_0}
	and the condition that for $q < 3/2$
	\begin{align} \label{eq:lambda_q}
		\lambda_{q} := \sup_{i} \sum_{j \neq i} \frac{R^3}{|X_i - X_j|^{2 q}} < \infty.
	\end{align}
\begin{thm}[{\cite[Corollary 2.7]{Hoefer19}}] \label{th:reflections}
	Let $1< r < 3 < q < \infty$.
	 Then, there exists $\bar c_0 > 0$ depending only on $q$ such that for all 
	$c_0 < \bar c_0$ defined as in \eqref{eq:c_0} and all $k \in \N$
	\begin{align}	\label{eq:L^infty.convergence}
	 	\|v^{(k)}_N - u_N\|_{L^\infty(\R^3)} \leq C (R^\alpha + \lambda_{q'}^{1/q'}) (C c_0)^k \|e v^{(0)}_N\|_{L^q(\cup_i B_i)},
	\end{align}
	where $\alpha = 1 - 3/q$ and $C$ depends only on $q$.
\end{thm}
Here, $q'$ is the H\"older dual of $q$, i.e. $\frac 1 q+\frac{1}{q'}=1$. At time $0$, assumption \eqref{eq:density.condition} and Lemma \ref {le:sums} imply that 
$c_0 + \lambda_q \le C \phi_N$. Moreover, we will see that $\|e v^{(0)}_N\|^q_{L^q(\cup_i B_i)} \leq \phi_N$ as well as $R^\alpha\le C\phi_N^{2}$. 
Thus, Theorem \ref{th:reflections} implies for $k =0,1$
\begin{align}
	\|v^{(k)}_N - u_N\|_{L^\infty(\R^3)} \leq  C (C \phi_N)^{k+1}
\end{align} 
as long as we control the minimal distance $\dmin$ sufficiently well.
In particular, if we are only interested in approximating $V_i$ up to terms of order $\phi_N$, as in the proof of Theorem \ref{th:tau}, it suffices to consider $v^{(0)}(X_i)$.
These values, we can compute explicitly. Indeed, the function $w_N$ has an explicit form, namely
\begin{align} \label{eq:w_n.explicit}
	w_N = \frac{g}{6 \pi N R} \quad \text{in } B_R(0), && w_N = \frac 1 N \Phi g - \frac {R^2}{6 N} \Delta \Phi g \quad \text{in } \R^3 \setminus B_R(0).
\end{align}
This leads to the approximation \eqref{eq:V_i.strategy1}.

\medskip

In order to obtain the refined approximation \eqref{eq:V_i.strategy2},
we need to consider $v^{(1)}_N(X_i)$. This function is not explicit anymore. However, the leading order term of $Q_i w_N$ is explicit.
Indeed, we have the following result taken from \cite{Hoefer19}. (Similar estimates can be found in \cite{NiethammerSchubert19,Gerard-VaretHillairet19, HillairetWu19}.)
\begin{lem}[{\cite[Lemmas 4.3 and 4.4]{Hoefer19}}] \label{lem:dipole}
Let \eqref{eq:separation.condition} be satisfied. Then, 
\begin{align}
	&\left|(Q_j w_N(\cdot - X_k))(X_i) - 5\frac{\phi_N}{N} e\Phi(X_i - X_j) \fint_{B_j} e w_N(x - X_k)  \dd x\right| \\ &\lesssim \frac{R^{5/2}}{|X_i - X_j|^3} \|e w_N(\cdot - X_k)\|_{L^2(B_j)}.
\end{align}
\end{lem}

Using in addition the explicit form of $w_N$ yields 
\begin{align} \label{eq:v^1}
	v^{(1)}_N(X_i) \approx \frac{g}{6 \pi N R}  + \sum_{j \neq i} \Phi(X_i - X_j)g - \frac{5 \phi_N}{N^2} \sum_{j\neq i} \sum_{k \neq j} e \Phi(X_i - X_j) e \Phi(X_j - X_k) g.
\end{align}
Using that we already know 
from Theorem \ref{th:tau} that the empirical measure $\rho_N$ is close to $\tau$, this leads to \eqref{eq:V_i.strategy2}.

\medskip

We emphasize that the last term on the right-hand side of \eqref{eq:v^1} needs to be handled very carefully.
Roughly speaking, this term is the main reason why we cannot treat the case 
that $\phi_N$ is small but non-vanishing. First, this term has a complicated structure, consisting of three-particle interactions, and second, the interaction kernel $e\Phi$ has the singularity $|e\Phi(x)| \sim |x|^{-2}$ which is critical in dimension $3$ (see also the next subsection).
Rigorous results on the derivation of mean-field limits with such singular interaction kernels are only known in special cases under some structural assumptions on the kernel, see \cite{SerfatyDuerinckx18}. These assumptions are not satisfied for $e\Phi$. To make matters worse, if one wants to treat $\phi_N$ of order $1$, one is forced to consider all terms in the expansion 
\begin{align*}
u_N := \lim_{k\to \infty}(1- \sum_i Q_i)^k v^{(0)}_N,
\end{align*}
which includes $k$-particle interactions for all $k$.

The main problem caused by the singular interaction regards the control of $\dmin$ as time evolves, which is needed for the results on the method of reflections, but also 
for the estimates of the Wasserstein distance discussed in the following section.
Due to the singularity of $e\Phi$ this leads to sums
\begin{align} \label{eq:alpha_3}
	\frac 1 N \sum_{j \neq i} \frac{1}{|X_i - X_j|^3},
\end{align} 
which we only control under the assumption \eqref{eq:ass.phi.logN}.
In fact, for controlling the particle distances, we rely on the following result from \cite{Hofer18MeanField}, which is also obtained through the method of reflections and assumptions \eqref{eq:density.condition} and \eqref{eq:ass.phi.logN}.
\begin{lem}[{\cite[Lemma 3.16]{Hofer18MeanField}}]
	For $k=2,3$, let
	\begin{align}
		\alpha_k = \sup_i \frac 1 N \sum_{j \neq i} \frac{1}{|X_i - X_j|^k}.
	\end{align}
		There exists $\delta > 0$ with the following property. 
		If \eqref{eq:separation.condition} is satisfied and $\phi_N \alpha_3 < \delta$, then
		\begin{align}
			|u_N(X_i) - u_N(X_j)| \leq C \alpha_2 |X_i - X_j|.
		\end{align}
\end{lem}

\subsection{Convergence to the mean-field limit by a generalization of a  result of \texorpdfstring{\cite{Hauray09}}{Hauray}}
\label{sec:meanfield}

For the second step, we prove the following theorem which generalizes a classical result of Hauray \cite{Hauray09}.
Hauray considers dynamical particle systems in $\R^d$ described by
\begin{align}
	\frac{\dd}{\dd t} X_i = \frac 1 N \sum_{j \neq i} K(X_i - X_j)
\end{align}
where the interaction kernel $K$ satisfies the condition
\begin{align} \label{eq:C_alpha} \tag{$C_\alpha$}
	\dv K = 0 , \quad \forall x \in \R^d ~ |K(x)| + |x| |\nabla K(x)| \leq \frac{C}{|x|^\alpha}
\end{align}
with $\alpha < d-1$.
Hauray shows in \cite[Theorem 2.1]{Hauray09} that the infinite Wasserstein 
distance between the empirical density and a continuous limit $\sigma$
is controlled by its distance at time zero for sufficiently well-prepared initial data, where the limit density solves
\begin{align}
	\partial_t \sigma + (K \ast \sigma) \cdot \nabla \sigma = 0, \quad \sigma(0) = \sigma_0.
\end{align}

More precisely, if $K$ satisfies the condition \eqref{eq:C_alpha} with $\alpha < d - 1$ and 
\begin{align} \label{eq:ass.Hauray}
	\lim_{N \to \infty} \frac{\left(\W_\infty(\rho_N(0),\sigma_0)\right)^d}{\dmin^{1+\alpha}} = 0,
\end{align}
then for all $T_\ast > 0$ and all $N$ sufficiently large (depending on $T_\ast$) 
\begin{align}
	\W_\infty(\rho_N(t),\sigma(t)) \leq \W_\infty(\rho_N(0),\sigma_0) e^{C \|\sigma_0\|_\infty t} \quad \text{for all } t \leq T_\ast.
\end{align}

\medskip
Clearly, we cannot directly apply this result in our setting where the particle velocities satisfy \eqref{eq:V_i.strategy1} and \eqref{eq:V_i.strategy2}. 
We therefore generalize the result of Hauray. In particular our result includes error terms like $E_i$ and $\bar E_i$ above as well as  additional external velocity fields like
$5 (e \Phi \ast (\tau (e \Phi g \ast \tau))$ or $\frac{g}{6\pi NR}$. Moreover, by a refined estimate (see Lemma \ref{lem:sums.Wasserstein}, we are able to relax assumption \eqref{eq:ass.Hauray}, which we will discuss in the remark after the statement of the theorem.

\begin{thm} \label{th:mean.field.general}
	Let $K \colon \R^d \to \R^d$ satisfy \eqref{eq:C_alpha} for some $\alpha < d -1$.
	Moreover, let $(\varphi_N)_{N\in \N}$ be a sequence of divergence-free vector fields uniformly bounded in $ L^\infty((0,\infty);W^{1,\infty}(\R^d))$.
	 Let $\tau_N^0 \in L^\infty(\R^d) \cap \mP(\R^d)$ be uniformly bounded, where $\mP$ denotes the space of probability densities
	and let $\tau_N$ be the solutions to 
	\begin{align}
\left\{\begin{array}{rl}
		\partial_t \tau_N + (K \ast \tau_N + \varphi_N) \cdot \nabla \tau_N &= 0, \\
		\tau_N(0) &= \tau_0^N.
\end{array}\right.
	\end{align}
	Let $X_i^0 \in \R^d$, $1 \leq i \leq N$, and consider
	\begin{align} \label{eq:sigma_N.ODE}		\dot X_i(t) = \varphi_N(X_i) + \frac 1 N \sum_{j \neq i} K(X_i(t) - X_j(t)) + E_i(t).
	\end{align}
		Let
	\begin{align}
		\sigma_N(t) := \frac 1 N \sum_i \delta_{X_i(t)}, &&
		\eta(t) := \W_\infty(\sigma_N(t),\tau_N(t)).
	\end{align}
	Assume there exists a sequence of non-decreasing functions $e_N$  such that $\eta$, $\dmin$ and $E_i$ satisfy for all $t \geq 0$
\begin{align} \label{eq:Wasserstein.T=0.generic}
		\lim_{N \to \infty} \frac{(\eta(0) +  e_N(t))^{d-(1 + \alpha)}}{\dmin(0)^{1+\alpha}N^{(1+\alpha)/d}} = 0, \qquad \dmin(0)N^{1/d}\ls 1, 
\end{align}
and
\begin{align}
		\forall \lambda > 0 \, \exists N_0> 0 \, \forall N > N_0 \quad \frac{\eta(t)}{\eta(0) + e_N(t)} + \frac{\dmin(0)}{\dmin(t)} \leq \lambda
		\implies 
		\left\{\begin{array}{rl}
			 \displaystyle \sup_i E_i(t) \leq e_N(t), \\
			 \displaystyle \sup_{i \neq j} \frac{|E_i(t) - E_j(t)|}{|X_i - X_j|} \leq C_1. 
		\end{array} \right.  \label{eq:E_i.bounds}
	\end{align}
	where $C_1$ depends only on $\sup_N \|\tau_0^N\|_{L^\infty}$, $K$, and $\sup_N \|\varphi_N\|_{L^\infty(W^{1,\infty})}$.
	Then, for all $T>0$ and all $N$ sufficiently large (depending on $T$) it holds for all $t\in [0,T]$:
	\begin{equation}
	\label{eq:eta.d_min.est}
	\begin{aligned}
		\eta(t) &\leq e^{C_2 t}\bra{\eta(0) + e_N(t)},\\
		\dmin(t) &\geq \dmin(0) e^{-C_2 t},
	\end{aligned}
	\end{equation}
	where $C_2$ depends only on $C_1$, $\sup_N \|\tau_0^N\|_{L^\infty}$, $K$, and $\sup_N \|\varphi_N\|_{L^\infty(W^{1,\infty})}$.
Moreover,
\begin{align} \label{eq:estimate.convolutions}
	|(K \ast ( \tau_N - \sigma_N)(x)| \leq \frac{C_2}{N |\dist(x,\{X_i\}_i)|^\alpha} + C_2 e^{C_2 t} (\eta(0) + e_N(t)).
\end{align}
\end{thm}
 \begin{rem} 
Since $\eta(0) \gtrsim N^{-1/d}$ by \eqref{eq:discretization.error}, this theorem implies \cite[Theorem 2.1]{Hauray09} by considering $\varphi_N = 0 = E_i$.
 On the other hand, assumption \eqref{eq:Wasserstein.T=0.generic} is considerably less stringent than the naive generalization of \eqref{eq:ass.Hauray}
 \begin{align}  \label{eq:ass.naive}
	\lim_{N \to \infty} \frac{\eta(0)^d + e_N(t)^d}{\dmin^{1+\alpha}} = 0.
\end{align}
Recall that for our purpose, the proof of Theorems \ref{th:tau} and \ref{th:main}, $e_N$ will include terms like $\phi_N$ and $\phi_N^2 |\log \phi_N|$. Thus, \eqref{eq:ass.naive} would impose a rate of the convergence $\phi_N \to 0$ much more severe than \eqref{eq:ass.phi.logN}. On the other hand, due to assumption \eqref{eq:separation.condition}, condition \eqref{eq:Wasserstein.T=0.generic} just requires $e_N \to 0$.

We also remark that, following \cite{CarilloChoiHauray14}, one can expect Theorem \ref{th:mean.field.general} to generalize to the case when $\rho \in L^\infty \cap \mP$ is replaced by $ \rho \in L^p \cap \mP$, if $\alpha < -1 + d/p'$. 
\end{rem}

Theorem \ref{th:mean.field.general} gives control over the infinite Wasserstein distance.
In order to control the $p$-Wasserstein distance, as stated in our main results,
we borrow from \cite{Mecherbet19}. We introduce the intermediate density
\begin{align} \label{eq:rho.bar}
	\bar \rho_N^0(x) = \frac 1 N \sum_i \psi \left(\frac{x - X_i^0}{\dmin}\right),
\end{align}
where $\psi$ is a standard mollifier. Note that assumption \eqref{eq:density.condition}
ensures that $\bar \rho_N^0$ is uniformly bounded in $L^\infty$. 
Thus, we will apply Theorem \ref{th:mean.field.general} to estimate $\W_\infty(\bar \tau_N(t), \rho_N(t))$ and $\W_\infty(\bar \rho_N(t), \rho_N(t))$, where $\bar \tau_N$ and $\bar \rho_N$ are the solutions to \eqref{eq:v.tau} and \eqref{eq:u.rho} respectively with initial data $\bar \rho_N^0$.
In a second step, we will then use stability of the systems \eqref{eq:v.tau} and \eqref{eq:u.rho} to estimate $\W_\infty(\bar \tau_N(t), \tau(t))$ and $\W_\infty(\bar \rho_N(t), \rho(t))$.

\subsection{From the mean-field limit to the effective evolution}
\label{sec:transition}

For the final step of the proof of Theorem \ref{th:main}, it remains to see that the solution to the mean-field limit \eqref{eq:u.rho} is close to the solution of the effective model \eqref{eq:rho_eff.u_eff}. This is provided by the following result.

\begin{prop} \label{pr:rho.rho_eff} 
	Let $\rho_0\in W^{1,1}\cap W^{1,\infty}\cap \fP$ and $1 \leq p < \infty$. Let $\rho_{\eff}$ and $\rho$ be the solutions to \eqref{eq:rho_eff.u_eff} and \eqref{eq:u.rho} respectively. Then, for every $T_\ast$ there exists a constant $C$ only depending on $p,T_\ast, \norm{\rho_0}_{W^{1,1}\cap W^{1,\infty}}$ such that
	\begin{align}\label{eq:prop_continuous}
		\W_p(\rho_\eff(t),\rho(t))+ \norm{(u_\eff-u)(t)}_{L^p}  &\leq C \phi_N^2.
	\end{align}
\end{prop}

The proof is based on the following two insights. First, by subtracting \eqref{eq:v.tau} and \eqref{eq:rho_eff.u_eff} we have 
\begin{align}  
\left\{\begin{array}{rl}
-\Delta (v-u_\eff)+\nabla p-(\tau-\rho_\eff) g&=5\phi_N\dv(\rho_\eff eu_\eff),\\
\partial_t \tau+(v+\bra{6 \pi \gamma_N}^{-1} g)\cdot \nabla \tau&=0,\\
\partial_t \rho_\eff+(u_\eff+\bra{6 \pi \gamma_N}^{-1} g)\cdot \nabla \rho_\eff&=0,\\
(\tau-\rho_\eff)(0)&=0.
\end{array}\right.
\end{align}
 The only source of difference between $(\tau,v)$ and $(\rho_\eff,u_\eff)$ is the source term on the right-hand side in the first equation. This term is at most of order $\phi_N$. Thus, we can expect that $\tau-\rho_\eff$ and $v-u_\eff$ are of order $\phi_N$. 

In a second step, we subtract \eqref{eq:u.rho} and \eqref{eq:rho_eff.u_eff} to obtain 
\begin{align}  
\left\{\begin{array}{rl}
-\Delta (u-u_\eff)+\nabla p-(\rho-\rho_\eff) g&=5\phi_N\dv(\rho_\eff eu_\eff-\tau ev),\\
\partial_t \rho+(u+\bra{6 \pi \gamma_N}^{-1} g)\cdot \nabla \rho&=0,\\
\partial_t \rho_\eff+(u_\eff+\bra{6 \pi \gamma_N}^{-1} g)\cdot \nabla \rho_\eff&=0,\\
(\rho-\rho_\eff)(0)&=0.
\end{array}\right.
\end{align}
We can argue as above to see that the only source of difference is the term on the right-hand side in the first equation. This time the right-hand side is of order $\phi_N^2$ where we used the first step. Thus we can expect $\rho-\rho_\eff$ and $u-u_\eff$ to be of order $\phi_N^2$.

We finish this section by stating the well-posedness results for systems \eqref{eq:v.tau}, \eqref{eq:u.rho} and \eqref{eq:rho_eff.u_eff} together with regularity results that will be used to  estimate the right-hand sides above.
\begin{thm} \label{th:hoefer} 
Assume that $\rho_0\in L^\infty\cap \fP$. There is a unique solution for all times to \eqref{eq:v.tau} and for all $3< p< \infty$:
\begin{align}
  \norm{\tau}_{L^\infty(L^\infty)}&\le \norm{\rho_0}_{L^\infty},\\
  \norm{v}_{L^\infty(W^{2,p})}&\le C \norm{\rho_0}_{L^\infty},
\end{align}
where $C$ only depends on $p$. 
If, in addition, $\rho_0 \in W^{1,1} \cap W^{1,\infty}$, then, for all $T_\ast > 0$
\begin{align}
  \norm{\tau}_{L^\infty(0,T_\ast; W^{1,1} \cap W^{1,\infty})}&\le C,\\
  \norm{v}_{L^\infty(W^{3,p})}&\le C,
\end{align}
where $C$ only depends on $p$, $T_\ast$ and $\norm{\rho_0}_{W^{1,1}\cap W^{1,\infty}}$.

Moreover, let $\tau_1,v_1$ and $\tau_2,v_2$ be two solutions of \eqref{eq:v.tau} corresponding to initial data $\tau_0^1,\tau_0^2\in L^\infty \cap \fP$. Then, for $1 \leq p\le \infty$:
 \begin{align*}
 \W_p(\tau_1(t),\tau_2(t))\le \W_p(\tau_0^1,\tau_0^2)e^{Ct},
 \end{align*}
 where $C$ only depends on $\norm{\tau_0^1}_{L^\infty}$ and $\norm{\tau_0^2}_{L^\infty}$,
  and for all $p> 3/2$ and $1/q \leq 1/ p + 1/3$
 \begin{align}
 	 \norm{u_1-u_2}_{L^p}\le C\W_q(\rho_0^1,\rho_0^2)e^{Ct},
 \end{align}
where $C$ depends on $\norm{\rho_0^1}_{L^\infty},\norm{\rho_0^2}_{L^\infty},p$.
\end{thm}

\begin{thm} \label{th:intermediate}
Assume that $\rho_0\in W^{1,1}\cap W^{1,\infty}\cap \fP$. For given $T_\ast>0$ a unique solution to \eqref{eq:u.rho} exists on $[0,T_\ast]$ and for all $3<p<\infty$:
\begin{align}
  \norm{\rho}_{L^\infty(0,T_\ast;W^{1,1}\cap W^{1,\infty})}&\ls C,\\
  \norm{u}_{L^\infty(0,T_\ast;W^{2,p})}&\ls C,
\end{align}
where $C$ depends on $T_\ast, \norm{\rho_0}_{W^{1,1}\cap W^{1,\infty}}$.

 Moreover, let  $\tau,v$ be determined by some $\rho_0\in W^{1,1}\cap W^{1,p}\cap \fP$ and let $\rho_1,u_1$ and $\rho_2,u_2$ be two solutions of \eqref{eq:u.rho}
corresponding to initial data $\rho_0^1,\rho_0^2\in W^{1,\infty}\cap W^{1,1}\cap \fP$. Then, for all $1 \leq p\le \infty$:
 \begin{align*}
 \W_p(\rho_1(t),\rho_2(t))\le \W_p(\rho_0^1,\rho_0^2)e^{Ct},
 \end{align*}
 where $C$ depends on $\norm{\rho_0^1}_{W^{1,1}\cap W^{1,\infty}},\norm{\rho_0^2}_{L^\infty}$,
 and for all $p> 3/2$ and $1/q \leq 1/ p + 1/3$.
 \begin{align}
 	 \norm{u_1-u_2}_{L^p}\le C\W_q(\rho_0^1,\rho_0^2)e^{Ct},
 \end{align}
where $C$ depends on $\norm{\rho_0^1}_{L^\infty},\norm{\rho_0^2}_{L^\infty},p$.
\end{thm}

\begin{thm} \label{thm:effective}
Assume that $\rho_0\in W^{1,1}\cap W^{1,\infty}\cap \fP$. There exists $\phi_0$ such that for all $\phi_N\le \phi_0$ and for given $T_\ast>0$ a unique solution to \eqref{eq:rho_eff.u_eff} exists on $[0,T_\ast]$ and such that for all $3< p<\infty$:
\begin{align}
  \norm{\rho_\eff}_{L^\infty(0,T_\ast;W^{1,1}\cap W^{1,\infty})}&\ls C,\\
  \norm{u_\eff}_{L^\infty(0,T_\ast;W^{2,p})}&\ls C,
\end{align}
where $C$ depends on $T_\ast, \norm{\rho_0}_{W^{1,1}\cap W^{1,\infty}}$.
\end{thm}

\begin{rem}
A well-posedness result for the system \eqref{eq:v.tau} was proven with related but slightly different spaces in \cite{Hofer18MeanField}. A very similar result to Theorem \ref{th:hoefer} has been recently shown in the parallel contribution \cite{Mecherbet20}.

Note that in contrast to the situation in Theorem \ref{th:hoefer}, we cannot expect that $u, u_\eff\in W^{2,p}$ when we only have $\rho,\rho_\eff\in L^\infty\cap \fP$ in Theorems \ref{th:intermediate} and \ref{thm:effective}, since $\tau$ and $\rho_\eff$, respectively, appear inside the divergence on the left-hand side of the equation.


The restriction $3<p$ comes from the fact that the velocity fields are not in $L^p$ for smaller $p$. Interestingly enough the restriction is less severe for the stability statement. This is due to the fact that all appearing densities have mass one so that their difference has mass zero. This leads to a stronger decay of the corresponding difference of the velocity field allowing for lower integrability. As mentioned before, the restriction $p<\infty$ comes from the lack of an $L^\infty$ theory for the Stokes equation.

Since all appearing continuous densities satisfy a transport equation, their $L^\infty$-norm is conserved. Also, since the corresponding velocity fields are divergence-free, the transport equation is at the same time a continuity equation which guaranties conservation of the $L^1$-norm and hence probability densities stay probability densities as they evolve.
\end{rem}

\section{Proof of Theorem \ref{th:mean.field.general}} 
\label{sec:meanfieldproof}

We begin this section by proving the following lemma which provides the key estimate to our improvement compared to \cite[Theorem 2.1]{Hauray09}.
It allows to estimate discrete convolutions with singular kernels by exploiting closeness of the empirical measure to a continuous density. This lemma will also be used in the next section. 
\begin{lem} \label{lem:sums.Wasserstein}
Let $X_i \in \R^d$, $1 \leq i \leq N$ and $\sigma_N = \frac 1 N \sum \delta_{X_i}$.
Let $\sigma \in \mP(\R^d) 
\cap L^\infty(\R^d) $, fix any of the particles $X_i$ and consider $\J := \{j \neq i \colon |X_j  - X_i| \leq 
\lambda \}$. Then, for all $\lambda \geq \W_\infty(\sigma_N,\sigma)$ and all  $\beta \in (0,d)$
\begin{align}\label{eq:sums.Wasserstein.close}
	\frac 1 N \sum_{j \in \J} \frac 1 {|X_i - X_j|^\beta} 
	\lesssim \frac{ \|\sigma\|_{L^\infty}^{(d-\beta)/d}\lambda^{d-\beta}}{N^{\beta/d} \dmin^{\beta}}.
\end{align}
Furthermore,
\begin{align} \label{eq:sums.Wasserstein}
	\sup_i \frac 1 N \sum_{i \neq j} \frac 1 {|X_i - X_j|^\beta} 
	\lesssim 1+\|\sigma\|_{ L^\infty} 
	+\frac{ \|\sigma\|_{L^\infty}^{(d-\beta)/d}(\W_\infty(\sigma_N,\sigma))^{d-\beta}}{N^{\beta/d} \dmin^{\beta}}.
\end{align}
\end{lem}
\begin{proof}
	 We use  the estimate
	 \begin{align} \label{eq:fractional.convolution}
	 	\||\cdot|^{-\beta} \ast \psi \|_{L^\infty} \lesssim \|\psi\|_{L^\infty}^{\beta/d}\|\psi\|_{L^1}^{(d-\beta)/d}
	 \end{align}
	 with 
	 \begin{align}
	 	\psi = \frac{2^d}{N \dmin^{d}\abs{B_1(0)}} \sum_{j \in \J} \1_{B_{\dmin/2}(X_j)}. 
	 \end{align} 
	 Note that $\abs{X_i-X_j}\le 2\abs{X_i-y}$ for all $y\in B_{\dmin/2}(X_j)$ and thus
	 \begin{align}
	 	\frac 1 N \sum_{j \in \J}|X_i - X_j|^{-\alpha}\ls \||X_i-\cdot|^{-\alpha} \ast \psi \|_{L^\infty}. 
	 \end{align}
	 Let $T$ be an optimal transport plan for $\sigma_N,\sigma$, i.e.,
a map $T \in L^\infty(\R^3)$ such that
$\sigma_N = T\# \sigma$ and 
\begin{align}
	\W_\infty(\sigma_N,\sigma) = \sigma - \esssup |T(x) - x|.
\end{align}
Such a plan exists if $\W_\infty(\sigma_N,\sigma) < \infty$ (see e.g. \cite{Santambrogio15}), otherwise the statement is trivial.
	There exists $x \in \R^3$ such that $T x = X_i$ and $X_{\J}\coloneqq\set{X_j:j\in \J} \subset T\bra{\overline{B_{\lambda +\W_\infty(\sigma_N,\sigma) }(x)}}$ and thus 
	 \begin{align}
	 \|\psi\|_{L^1}=\sigma_N(X_{\J}) \leq \sigma(\overline{B_{2 \lambda}(x)})\ls \|\sigma\|_\infty \lambda^d.
	 \end{align}
	 This yields \eqref{eq:sums.Wasserstein.close}.
	 
	 \medskip
	 
	 To prove \eqref{eq:sums.Wasserstein}, we fix again $X_i$ and apply  \eqref{eq:sums.Wasserstein.close} with $\lambda = 2 \W_\infty(\sigma_N,\sigma)$.
	 We need to estimate the remaining sum over $\tilde \J := \{j \colon X_j \neq X_i, |X_j - X_i| >  2 \W_\infty(\sigma_N,\sigma) \}$. Then, for $j \in \tilde J$, we use $|X_j - X_i| \geq \frac 1 2 |y - X_i|$ for $\sigma$-almost every
	 $y$ such that $T y = X_j$. Thus, with $U = T^{-1}(\{ X_k \colon k\in \tilde J\})$,
	\begin{align}
		 \frac 1 N  \sum_{k  \in \tilde J}  \frac{1}{|X_i - X_k|^{ \beta}} =  \int_{U}  \frac{1}{|X_i -T y|^{\beta}} \sigma(y) \dd y 
		  \leq C \int_{U}  \frac{1}{|X_i - y|^{\beta}} \sigma(y) \dd y  \leq C(1+\|\sigma\|_{L^\infty}).
	\end{align}	
	This concludes the proof.
\end{proof}

For the proof of Theorem \ref{th:mean.field.general} and later on, we use the standard technique to 
express the solutions to  
transport equations by  flow maps in order to estimate the Wasserstein distance between two solutions.

More precisely, consider a transport equation
 \begin{align} \label{eq:transport}
	\partial_t \sigma+u\cdot \nabla \sigma=0, \qquad \sigma(0) = \sigma_0,
\end{align}
with $\dv u = 0$.
Then, we can write the solution
\begin{align}
	\sigma(t,x) = \sigma_0(Y(0,t,x))
\end{align}
where
\begin{align}
	\partial_t Y(t,s,x) = u(t,Y(t,s,x)), \quad Y(t,t,x) = x.
\end{align}
This is possible both for the continuous systems that we consider, where $\sigma_0 \in L^\infty$ and $u \in L^\infty(0,t_\ast;W^{1,\infty})$, and for the discrete systems until the first collision of particles. For the discrete system coupled through the Stokes equations, \eqref{eq:u_N}--\eqref{eq:def.V_i}, well-posedness until the first collision has been proved in \cite[Theorem A.1]{Hofer18MeanField}.

Consider now $(\sigma_i,u_i)$ which solve the transport equation \eqref{eq:transport} on $[0,t_\ast)$ and are given through flow maps $Y_1,Y_2$.
Assume $\sigma_i \in \mP(\R^d)$ and that $\sigma_1$ is continuous with respect to the Lebesgue measure and $\sigma_1 \in L^\infty(\R^d)$ (which is conserved in time). Then, for any time ${\red t_0 } \in [0,t_\ast)$, there exists an optimal transport plan $T$ such that
 \begin{align}
	\W_\infty(\sigma_1(t_0),\sigma_2(t_0))= \sigma_1(t_0) - \esssup_x \abs{T(x)-x}, \quad \text{ for } p=\infty,\\
	\W_p(\sigma_1(t_0),\sigma_2(t_0))= \bra{\int_{\R^d} \abs{T(x)-x}^p \sigma_1(t_0,x)\dd x}^{1/p}, \quad \text{ for } p<\infty.
\end{align}
Consider $T_t=Y_2(t,t_0,\cdot)\circ T\circ Y_1(t_0,t,\cdot )$. $T_t$ is well-defined because $T$ maps $\sigma_1(t_0)$-a.e. point into the support of $\sigma_2$. $T_t$ is a transport plan for $(\sigma_1(t),\sigma_2(t))$, i.e. $\sigma_2(t) = T_t\# \sigma_1(t)$, by the property of the flow maps.

This leads to the following well-known result which we prove here for self-containedness:
\begin{lem}\label{lem:Wass.deriv}
Let $(\sigma_i,u_i)$, $i=1,2$, $t_0$ and $T_s$ as above and consider
\begin{align}\label{eq:fdef}
	f(t)\coloneqq \sup_{t_0\le s\le t} \sigma_1(s)-\esssup |T_s(x) - x|,
\quad \text{ for } p=\infty,\\
	f(t)\coloneqq \sup_{t_0\le s\le t} \bra{\int_{\R^d} \abs{T_s(x)-x}^p \sigma_1(s,x)\dd x}^{1/p}, \quad \text{ for } p<\infty.
\end{align}

Then, for all $t_0\le t_1\le t_2<t_\ast$
 \begin{align}
	f(t_2) - f(t_1) \le \int_{t_1}^{t_2}\sigma_1(s) - \esssup_x \abs{(u_2(s)\circ T_s)(x)-u_1(s,x)} \dd s, \quad \text{ for } p=\infty,\\
	f(t_2) - f(t_1)\le \int_{t_1}^{t_2} \bra{\int_{\R^d}\abs{(u_2(s)\circ T_s)(x)-u_1(s,x)}^p\sigma_1(s,x)\dd x}^{1/p} \dd s, \quad \text{ for } p<\infty.
\end{align}
\end{lem}

\begin{rem}
Note that we have $f(t_0)=\W_p(\sigma_1(t_0),\sigma_2(t_0))$ and for all $t>t_0$:
 \begin{align}\label{eq:f.larger.Wass} 
	\W_p(\sigma_1(t),\sigma_2(t))\le f(t).
\end{align}
\end{rem}

\begin{proof}
We first consider the case $p = \infty$. We need to estimate $\abs{T_t(x)-x}$ for general $x$. For the position of a generic particle in the continuous system found at $x_t\in \supp \sigma_1(t)$ at time $t$ we use the notation
\begin{align}
 x_s=Y_1(s,t,x_t).
\end{align}
We have for $t>t_0$:
\begin{equation}\label{eq:Tx.x}
\begin{aligned}
  T_t(x_t)-x_t&=Y_2(t,t_0,T(x_{t_0}))-Y_1(t,t_0,x_{t_0})\\
  &= T(x_{t_0})-x_{t_0}+\int_{t_0}^t \partial_t Y_2(s,t_0,T(x_{t_0}))-\partial_t Y_1(s,t_0,x_{t_0})\dd s\\
   &= T(x_{t_0})-x_{t_0}+\int_{t_0}^t u_2(s,T_s(x_s))-u_1(s,x_s)\dd s.
\end{aligned}
\end{equation}
Taking the $\sup$ we conclude
\begin{equation}
  f(t)\le f(t_0)+ \int_{t_0}^t \sigma_1(s)  - \esssup |u_2(s)\circ T_s - u_1(s)|\dd s.
\end{equation}
If $p<\infty$ we define $f$ analogously to \eqref{eq:fdef} with the $\infty$ distance replaced by the $p$ distance. Using \eqref{eq:Tx.x} we estimate
\begin{align}
f(t)&\le \bra{\int_{\R^3}\abs{T_t(x_t)-x_t}^p\sigma_1(t,x_t)\dd x_t}^{1/p}\\
&\le \bra{\int_{\R^3}\abs{T(x_{t_0})-x_{t_0}}^p\sigma_1(t_0,x_{t_0})\dd x_{t_0}}^{1/p} \\
& +\int_{t_0}^t \bra{\int_{\R^3}\abs{u_2(s,T_s(x_s))-u_1(s,x_s)}^p\sigma_1(s,x_s)\dd x_t}^{1/p}\dd  s.
\end{align}
Again, taking the $\sup$, we conclude.
\end{proof}

In the proof of Theorem \ref{th:mean.field.general} and in the rest of the paper, we will use the following estimate for kernels that satisfy \eqref{eq:C_alpha} for arbitrary $\alpha$ (in particular they hold for the fundamental solution $\Phi$ and its gradient $\nabla \Phi$ given in \eqref{eq:Phi} and for $\Delta \Phi$):
\begin{align}
	|K(x) - K(y)| \leq C |x-y| \left( \frac 1 {|x|^{1+\alpha}} + \frac 1 {|y|^{1+\alpha}} \right). \label{eq:Phi.Lipschitz} 
\end{align}

For the proof of Theorem \ref{th:mean.field.general} we follow the proof of \cite[Theorem 2.1]{Hauray09}. There, the proof proceeds in three steps:
	The first two steps consist in proving the differential inequalities
	\begin{align}
		\frac \dd {\dd t} f &\ls f ( 1 + f^{d-1} \dmin^{-\alpha}), \label{eq:eta.Hauray}\\
		\frac \dd {\dd t} \dmin &\gs -  \dmin ( 1 + f^{d} \dmin^{-(1+ \alpha)}) \label{eq:dmin.Hauray}.
	\end{align}
	The third step is the conclusion using  the analogous assumption to \eqref{eq:Wasserstein.T=0.generic} and standard ODE-theory.	
	 For the sake of completeness, we provide the full proof highlighting the necessary adaptations.

\begin{proof}[Proof of Theorem \ref{th:mean.field.general}] 
The proof relies on application of Lemma \ref{lem:Wass.deriv} to $\tau_N(s),\sigma_N(s)$. We recall that the dynamics of the empirical measure $\sigma_N$ is governed by the ODE system \eqref{eq:sigma_N.ODE}. Thus, the evolution is not a priori of the form of a transport \eqref{eq:transport} and Lemma \ref{lem:Wass.deriv} not directly applicable. It would not be difficult to adapt Lemma \ref{lem:Wass.deriv} to cover this case as well. Alternatively, one might bring the evolution of $\sigma_N$ into the form of \eqref{eq:transport}. To this end, one just needs to find a (divergence free) vector field $E(t,x)$ such that $E(t,X_i(t)) = E_i(t)$. Such a vector field exists, at least until the time of the first collision of particles (here collision really means $X_i(t) = X_j(t)$ for some $i \neq j$). We will never go beyond this time, and we will never evaluate $E$ anywhere outside the particle positions. In particular, the exact choice of $E$ is irrelevant.

Let now $t_0=0$ and $f$ as in Lemma \ref{lem:Wass.deriv} with $p=\infty$. Then we have:
	\begin{align}
		&f(t)-f(0)\leq   \int_0^t\sup_i E_i(s)\\
		 & + \tau_N(s)  - \esssup_x \left| \int_{\R^d} (K(T_s(x) - T_s(y)) - K(x-y)) \tau_N(s,y) \dd y +\varphi_N(T_s(x))-\varphi_N(x))\right|\dd s,
	\end{align}	
	where $T_s$ is the transport plan from Lemma \ref{lem:Wass.deriv} for $\tau_N(s),\sigma_N(s)$.	As in \cite{Hauray09}, we split the spatial integral into two parts.
	 Omitting the time variable, we denote $J_1 := \{ y : |x-y| \geq 4 f\}$.
	 Then, (for $\tau_N$-almost every  $x$) we use that for all $y \in J_1$
	 \begin{align}
	 	|K(T(x) - T(y)) - K(x-y))| \leq \frac{C f}{\min\{|T(x) - T(y)|^{1+\alpha}, |x-y|^{1+\alpha}\}} \leq \frac{C f}{|x-y|^{1+\alpha}},
	 \end{align}
	 where we used \eqref{eq:Phi.Lipschitz} and the fact that
	 \begin{align}\label{eq:triangle}
	 	|T(x) - T(y)|\ge \abs{x-y}-\abs{x-T(x)}-\abs{y-T(y)}\ge \abs{x-y}-2f\ge \frac 1 2 \abs{x-y}.
	 \end{align}
	 Hence,
	 \begin{align}
		I_1 =  \left| \int_{J_1} (K(Tx - Ty) - K(x-y)) \tau_N(t,y) \dd y \right| \leq C f(1+\|\tau_N^0\|_{L^\infty}).
	 \end{align}
	To estimate the remainder, $I_2$, the integral over the set $J_2 = \R^d \setminus J_1$, we proceed differently from \cite{Hauray09}.
	  Indeed, introducing $\J := \{i \colon X_i \neq T(x), |T(x) - X_i| \leq 6 f\}$, and observing that by a similar computation as \eqref{eq:triangle} $T(y)\in \J$ for all $y\in J_2$ we have
	\begin{align}
		I_2 =\left | \int_{J_2} (K(Tx - Ty) - K(x-y)) \tau_N(y) \dd y \right |
		& \lesssim \|\tau_N^0\|_\infty f^{d-\alpha} + \frac 1 N \sum_{i \in \J}|Tx - X_i|^{-\alpha}.
	\end{align}		 
	 Therefore, by Lemma \ref{lem:sums.Wasserstein},
	 \begin{align}
	 	I_2 
	 	\lesssim \|\tau_N^0\|_\infty f^{d-\alpha} + N^{-\alpha/d} \dmin^{-\alpha} f^{d - \alpha} \|\tau_N^0\|_{L^\infty}^{(d-\alpha)/d}.
	 \end{align} 	 
	 Thus, the inequality for $f$ becomes 
	 \begin{align}
	 	f(t)-f(0) \le C \int_0^t f \left( 1 + f^{d- (1 +\alpha)} \dmin^{-\alpha} N^{-\alpha/d}\right)  + \sup_i |E_i|\dd s,
	 \end{align}
	  where we put dependencies on $\|\tau_N^0\|_{L^\infty}$ and $\sup_N \|\varphi_N\|_{L^\infty(W^{1,\infty})}$ into the constant.
	 Note that the additional velocity field $\varphi_N$ does not affect this estimate since its influence is a linear term in $f$ on the right-hand side.
	 
	 We also observe, by the same argument, that we can estimate
	 \begin{align}
	 	|(K \ast \tau_N)(x) - (K \ast \sigma_N)(x)| \leq \frac{C}{N |\dist(x,\{X_i\}_i)|^\alpha} +  C f \left( 1 + f^{d- (1 +\alpha)} \dmin^{-\alpha} N^{-\alpha/d}\right),
	 \end{align}
	where the additional term on the right-hand side is due to the closest particle to $x$ which needs to be estimated separately. This will yield \eqref{eq:estimate.convolutions} once \eqref{eq:eta.d_min.est} is established.
	 
	Next, we use that for $i \neq j$
	\begin{align}
		&\frac 1 N \left(\sum_{k \neq i} K(X_i - X_k)  - \sum_{k \neq j} K(X_j - X_k)) \right) \\	
		&\lesssim |X_i - X_j| \frac 1 N  \sum_{k \not \in \{i,j\}}  
		\left( \frac{1}{|X_i - X_k|^{1+ \alpha}} +  \frac{1}{|X_j - X_k|^{1+ \alpha}} 	\right) + \frac{1}{N \dmin^\alpha}.
	\end{align}	 
	Thus 
	\begin{align}
		\frac \dd { \dd t} \dmin \gtrsim - \dmin \bra{1+\frac{1}{N \dmin^{1+\alpha}}
		+ \sup_{ i} \frac 1 N  \sum_{k \neq i}  \frac{1}{|X_i - X_k|^{1+ \alpha}} 
		+\sup_{i \neq j} \frac{|E_i - E_j|}{|X_i - X_j|}},
	\end{align}		
	 where the linear term in $\dmin$ comes from the function $\varphi_N$.  
	 Thus, relying again on Lemma \ref{lem:sums.Wasserstein},
	\begin{align}
		\frac \dd {\dd t} \dmin \geq - C  \dmin \left(1+ \frac{1}{N \dmin^{1+\alpha}}+ f^{d - (1+\alpha)} \dmin^{-(1+ \alpha)}N^{-(1+\alpha)/d} + \sup_{i \neq j} \frac{|E_i - E_j|}{|X_i - X_j|}\right).  
	\end{align}	
We set $T_\ast(N)$ the maximal time for which
	\begin{align}\label{eq:buckling}
	\begin{aligned}
		\eta^{d- \alpha - 1} \dmin^{-(1+\alpha)} N^{-(1+\alpha)/d} +\sup_{i \neq j} \frac{|E_i - E_j|}{|X_i - X_j|} &\leq 2C_1, \qquad \eta^{d- \alpha - 1} \dmin^{-\alpha} N^{-\alpha/d} \leq C_1,\\
		 \qquad N^{-1}\dmin^{-(1+\alpha)}&\le C_1, \qquad  \sup_i |E_i| \leq e_N .
	\end{aligned}
	\end{align}		
	
	Now let $t_\ast$ be the time until which inequalities \eqref{eq:buckling} are satisfied when $\eta$ is replaced by $f$. Then, until $t_\ast$, $f$ and $\dmin$ satisfy the following inequalities:
\begin{align}
f(t)-f(0)&\le C \int_0^t f \left( 1 + C_1\right)  + e_N\dd s,\\
		\frac \dd {\dd t} \dmin &\geq  - C \dmin \left( 1 +3C_1\right).  
	\end{align}	
A Gronwall argument gives the desired estimates until $t_\ast$, where we used that $\eta\le f$. If $t_\ast<T_\ast(N)$ we can start with an $f$ corresponding to an optimal transportation at $t_0=t_\ast$ in Lemma \eqref{lem:Wass.deriv} to expand the validity of estimate \eqref{eq:eta.d_min.est} over $t_\ast$. By repeating this argument, the interval in which the estimate holds is open and closed in $[0,T_\ast(N)]$ and thus \eqref{eq:eta.d_min.est} holds on $[0,T_\ast(N)]$. 
	It remains to prove $T_\ast(N) \to \infty$ as $N \to \infty$. We start with the observation that $\eta(0)\gs  N^{-1/d}$ (see \eqref{eq:discretization.error}) implies by \eqref{eq:Wasserstein.T=0.generic} that $N^{-1}\dmin(0)^{-{1+\alpha}}\to 0$ as $N\to \infty$.  For any $t>0$ we apply \eqref{eq:E_i.bounds} with $\lambda=2e^{C_2 t}$ to get an $N_0$ such that we have the right error bounds for $N\ge N_0$. We then have the following for $s<\min(t,T_\ast(N))$:
	\begin{align}
		\eta^{d- \alpha - 1} \dmin^{-(1+\alpha)} N^{-(1+\alpha)/d} +\sup_{i \neq j} \frac{|E_i - E_j|}{|X_i - X_j|} \leq \frac{(\eta(0) +  e_N)^{d-(1 + \alpha)}}{\dmin(0)^{1+\alpha}N^{(1+\alpha)/d}}e^{dC_2s}+C_1, \\
		\eta^{d- \alpha - 1} \dmin^{-\alpha} N^{-\alpha/d} \leq \frac{(\eta(0) +  e_N)^{d-(1 + \alpha)}}{\dmin(0)^{1+\alpha}N^{(1+\alpha)/d}}\dmin(0) N^{1/d}e^{(d-1)C_2s},\\
		N^{-1}\dmin^{-(1+\alpha)}\le N^{-1}\dmin(0)^{-(1+\alpha)} e^{(1+\alpha)C_2t},\qquad  \sup_i |E_i| \leq e_N .
	\end{align}
	Because of \eqref{eq:Wasserstein.T=0.generic} we can choose $N$ so large that the right hand sides are smaller than $2C_1$ and $C_1$, respectively, for $s=t$, which implies $T_\ast(N)\ge t$. Since $t$ was arbitrary this finishes the proof.
%
%
\end{proof}
\section{Proof of the main results}
\label{sec:proof.main}

We recall that from now on the dimension is again fixed to $d=3$. The following lemma is used to estimate some recurring  sums. We refer to \cite[Lemma 2.1]{JabinOtto04} and \cite[Lemma 4.8]{NiethammerSchubert19} for the proof.

\begin{lem} \label{le:sums}
For $k = 1,2$,
\begin{align} \label{eq:sum.1.2}
	\alpha_k := \sup_i  \frac 1 N\sum_{i \neq j} \frac 1 {|X_i - X_j|^k} \lesssim \frac{1}{N^{k/3}\dmin^{k}}.
\end{align}
Moreover, 
\begin{align} \label{eq:sum.3}
	\alpha_3 := \sup_i  \frac 1 N\sum_{i \neq j} \frac 1 {|X_i - X_j|^3} \lesssim \frac{\log N}{N\dmin^{3}} .
\end{align}
\end{lem}

\subsection{Proof of Theorem \ref{th:tau}}


%

\begin{proof}[Proof of Theorem \ref{th:tau}]
Let $\bar \tau_N$ be the solution to \eqref{eq:v.tau} with initial data $\bar \rho_N^0$ from \eqref{eq:rho.bar}. Then, we claim for all $T_\ast>0$ and all $N=N(T_\ast)$ sufficiently large, for all $t \leq T_\ast$
\begin{align} \label{eq:rho_N.tau.bar}
  \W_\infty (\rho_N(t),\bar \tau_N(t) )\leq C( \phi_N + \W_\infty(\rho_N(0), \bar \rho_N^0)) e^{C t}.
\end{align}
We first show how  the estimate for $\W_p (\rho_N(t),\tau(t) ) $ in \eqref{eq:rho_N.tau} follows from \eqref{eq:rho_N.tau.bar}:
We estimate
\begin{align}\label{eq:rhoN.tau}
	  \W_p (\rho_N(t),\tau(t) ) 
	  \leq \W_\infty (\rho_N(t),\bar \tau_N(t) ) 
	  + \W_p (\bar \tau_N(t),\tau(t) ),
\end{align}
where we used that for probability densities $\sigma_1,\sigma_2$ and $q\le r$ we have:
\begin{align}\label{eq:W.hirarchy}
  \W_q(\sigma_1,\sigma_2)\le \W_r(\sigma_1,\sigma_2)\norm{\sigma_1}_{L^1}^{(r-q)/r}=\W_r(\sigma_1,\sigma_2).
\end{align}
Moreover, by the stability result from Theorem \ref{th:hoefer},
\begin{align}\label{eq:tau_bar.tau}
	\W_p (\bar \tau_N(t),\tau(t)) \lesssim e^{Ct}  \W_p (\bar \rho_N^0,\rho_0) \leq e^{Ct}  \W_\infty (\bar \rho_N^0,\rho_N(0)) + e^{Ct}  \W_p ( \rho_N(0),\rho_0).
\end{align}
By definition of $\bar \rho_N^0$, we have
\begin{align} \label{eq:rho.bar.dmin}
	\W_\infty (\bar \rho_N^0,\rho_N(0)) \approx \dmin(0) \approx N^{-1/3},
\end{align}
where $\approx$ is used to indicate that both $\ls$ and $\gs$ hold. In view of \eqref{eq:discretization.error}, this gives
\begin{align} \label{eq:W_p.intermediate}
	 \W_\infty (\bar \rho_N^0,\rho_0) \lesssim   \W_p ( \rho_N(0),\rho_0).
\end{align}
Inserting \eqref{eq:W_p.intermediate} into \eqref{eq:rho_N.tau.bar} and \eqref{eq:tau_bar.tau} and the results into \eqref{eq:rhoN.tau} yields
 \eqref{eq:rho_N.tau}.

In order to prove the claim \eqref{eq:rho_N.tau.bar}
as well as \eqref{eq:min.dist.conserved}, it suffices to show that the assumptions of Theorem \ref{th:mean.field.general} are satisfied with $K = \Phi g$, 	$\varphi_N = \frac{g}{6 \pi N R}$ and	$e_N = C (\phi_N + R)$ where $\eta(t)=\W_\infty(\rho_N(t),\bar \tau_N(t))$.  Note that \eqref{eq:discretization.error} implies that 
$ R \ll \dmin(0) \lesssim C N^{-1/3} \lesssim C \W_\infty (\rho_0,\rho_N(0))$, and therefore the error $R$ does not appear in \eqref{eq:rho_N.tau}.

	Condition \eqref{eq:Wasserstein.T=0.generic} follows directly from assumptions \eqref{eq:density.condition}, \eqref{eq:ass.phi.logN} and \eqref{eq:rho.bar.dmin}.

	To prove \eqref{eq:E_i.bounds}  we rely on the results from \cite{Hofer18MeanField}.
	In \cite{Hofer18MeanField} rotations are neglected. However,
	the  results are easily adapted to include the rotations.
	We refer to \cite{NiethammerSchubert19} and \cite{Hoefer19} for 
	the details about  the necessary modifications of the method of reflections 	to include particle rotations.

	In \cite{Hofer18MeanField}, the first order approximation of the method of reflections,
	that we denote here by $v_N^{(0)}$, is denoted by $u$. Thus, \cite[Proposition 3.12]{Hofer18MeanField} yields that as long as
	$\phi_N \alpha_3 < \delta$ for some given $\delta >0$
	\begin{align} \label{eq:u_N.v^0}
		\|u_N - v_N^{(0)}\|_{L^\infty}  \lesssim \alpha_2 (\alpha_2 \phi_N + R) \lesssim (\alpha_2^2+1) (R+ \phi_N).
	\end{align}
	We resort to the explicit form of $v_N^{(0)}$ in \eqref{eq:def.v^k} to find 
	\begin{align} 
			\dot X_i &=	\frac 1 N \sum_{j \neq i} \Phi(X_i - X_j) g + \frac{g}{6 \pi N R} + E_i,  \\  E_i &:= \dot X_i - v_N^{(0)}(X_i) + \sum_{j \neq i} \left(w_N(X_i - X_j) - \frac 1 N \Phi(X_i - X_j) g\right),
\end{align}
where we used $w_N(0) = \frac{g}{6 \pi N R}$. 
The explicit form of $w_N$ from \eqref{eq:w_n.explicit} yields
\begin{align} \
	E_i &= u_N(X_i) - v_N^{(0)}(X_i) - \frac{R^2}{6N}\sum_{j \neq i} \Delta\Phi(X_i - X_j) g.
\end{align}
We have
\begin{align}\label{eq:DeltaPhi}
	\frac {R^2} N \sum_{j \neq i} |\Delta \Phi(X_i - X_j)| &\lesssim {R} \frac 1 N \sum_{j \neq i} |X_i - X_j|^{-2} = \alpha_2 R.
\end{align}
Thus,  as long as
	$\phi_N \alpha_3 < \delta$
\begin{align} \label{eq:est.E_i}
	|E_i| \lesssim (\alpha_2^2+1) (R+ \phi_N).
\end{align}
	
	Moreover, using $\Phi(X_i - X_j) = \Phi(X_j - X_i)$, we have
	\begin{align} 
		|E_i - E_j| &\leq |\dot X_i - \dot X_j| + \frac{1}{N}\left|\sum_{k \not \in \{j,i\}} \Phi(X_i - X_k)g - \Phi(X_j - X_k) g\right| \\
		& \lesssim |u_N(X_i) - u_N(X_j)| + \alpha_2 |X_i  - X_j|
	\end{align}
	by \eqref{eq:Phi.Lipschitz}.
	By \cite[Lemma 3.16]{Hofer18MeanField}, we have 
	\begin{align}
	|\dot X_i - \dot X_j|\ls \alpha_2\abs{X_i-X_j},
	\end{align}	
	 as long as $\phi_N \alpha_3 < \delta$ yielding
	\begin{align}\label{eq:est.E_i-E_j}
		|E_i - E_j| \lesssim \alpha_2 |X_i  - X_j|.
	\end{align}	
	Note that the important object to estimate in order to control the particle distance is $|\dot X_i - \dot X_j|$, which can be handled directly by \cite[Lemma 3.16]{Hofer18MeanField}. However, the separate estimate for $|E_i - E_j|$ is needed in order to apply Theorem \ref{th:mean.field.general}.
	
	To conclude the proof of the claim, we first observe that by Lemma \ref{le:sums}, and assumptions \eqref{eq:density.condition} and \eqref{eq:ass.phi.logN} as long as $\frac{\dmin(0)}{\dmin(t)} \le \lambda$:
	\begin{align}
		\alpha_3 \phi_N \lesssim \frac{\log N}{N \dmin(t)^3}\phi_N\ls \left( \frac{\dmin(0)}{\dmin(t)} \right)^3 \log N \phi_N \le \lambda^3 \log N \phi_N \to 0, \text { for }N\to \infty.
	\end{align}
	
	Thus, as long as $\frac{\dmin(0)}{\dmin(t)} \le \lambda$ and for $N$ sufficiently large, \eqref{eq:est.E_i} and \eqref{eq:est.E_i-E_j} hold and if additionally $\frac{\eta(t)}{\eta(0) + e_N}\le \lambda$, then by \eqref{eq:sums.Wasserstein}
	\begin{align}
		\alpha_2 \lesssim 1 + \frac{\eta(t)}{N^{2/3} \dmin(t)^{2}} \lesssim 1 + (\eta(0) + e_N) \frac{\eta(t)}{\eta(0) + e_N}  \left(\frac{\dmin(0)}{\dmin(t)}\right)^2\ls 1 + (\eta(0) + e_N) \lambda^3.
	\end{align}
	Since $\eta(0)+e_N\to 0$ as $N\to \infty$ by assumptions \eqref{eq:density.condition} and \eqref{eq:ass.phi.logN} as well as by \eqref{eq:rho.bar.dmin}, this shows that \eqref{eq:E_i.bounds}  is satisfied.
	
	Application of Theorem \ref{th:mean.field.general} yields the claim as well as 
	\eqref{eq:min.dist.conserved}.

	 It remains prove \eqref{eq:tau.v-u_N}. 
	 By the stability result from Theorem \ref{th:hoefer} and $\W_p(\bar \rho_N^0,\rho_0)\ls \W_p(\rho_N^0,\rho_0) $ (by \eqref{eq:W_p.intermediate}), it suffices to estimate 
	 $\|\bar v_N - u_N\|_{L^q_\loc}$ where $\bar v_N$ is the fluid velocity in the solution to \eqref{eq:v.tau} with initial data $\bar \rho_N^0$. 
	 Then, $\bar v_N =  \Phi g \ast \bar \rho_N$.
	 Thus, 
	 \begin{align}
	 	|\bar v_N - u_N| \leq |v_N^{(0)} - u_N| + | v_N^{(0)} - \Phi g \ast \rho_N| + |\Phi g \ast (\bar \rho_N - \rho_N)|.
	 \end{align}
	 The first term on the right-hand side, we control by \eqref{eq:u_N.v^0},
	 and the third term by \eqref{eq:estimate.convolutions}. The second term is estimated similarly as in \eqref{eq:DeltaPhi} by
	 \begin{align} \label{v^0.convolution}
	 	| v^{(0)} - \Phi g \ast \rho_N|(x) \lesssim \frac 1 {N |\dist(x,\{X_i\}_i)|} + R \alpha_2.
	 \end{align}
	 Thus, 
	 	 \begin{align}\label{eq:vN.uN}
	 	|\bar v_N - u_N|(x) & \lesssim \frac 1 {N |\dist(x,\{X_i\}_i)|} + (\alpha_2^2+1+Ce^{Ct}) (R+ \phi_N) + \W_\infty(\rho_N(0),\bar \rho_N^0) e^{Ct}. \qquad
	 \end{align}
	 Combining Lemma \ref{le:sums}, \eqref{eq:eta.d_min.est} and \eqref{eq:density.condition} we have $\alpha_2 \lesssim e^{Ct}$. By \eqref{eq:separation.condition}, \eqref{eq:density.condition}, \eqref{eq:discretization.error} and \eqref{eq:W_p.intermediate}, we have $R + N^{-1/3} + \W_\infty(\rho_N(0),\bar \rho_N^0) \lesssim  \W_p(\rho_N(0), \rho_0)$. Finally, 
	 \begin{align} \label{eq:int.dist.X_i}
	 	\||\dist(x,\{X_i\}_i)|^{-1}\|_{L^q_\loc} \lesssim \dmin^{-1}
	 \end{align}
	 for all $q < 3$ . Inserting the last three estimates into \eqref{eq:vN.uN} and using again \eqref{eq:discretization.error} yields
	 \begin{align}
	 	\|\bar v_N - u_N\|_{L^q_\loc} \lesssim \frac 1 N \dmin^{-1}+(\phi_N + \W_p(\rho_N(0), \rho_0)) e^{Ct} \lesssim  (\phi_N + \W_p(\rho_N(0), \rho_0)) e^{Ct}.
	 \end{align}
	This finishes the proof.
\end{proof}
\subsection{Proof Theorem \ref{th:main}}\label{sec:main_proof}

\begin{proof}[Proof of Theorem \ref{th:main}]
We follow the same argument as for the proof of Theorem \ref{th:tau}.
Let $\bar \rho_N$ be the solution to \eqref{eq:u.rho} with initial data $\bar \rho_N^0$ from \eqref{eq:rho.bar}. 
Then, we claim for all $T_\ast>0$ and all $N=N(T_\ast)$ sufficiently large, for all $t \leq T_\ast$
	\begin{align}
		&\W_\infty(\rho_N(t),\bar \rho_N(t)) \\
		&\leq\left( \W_\infty(\rho_N(0), \bar \rho_N^0) + \phi_N^2 |\log \phi_N| + \phi_N \W_\infty(\rho_N(0),\bar \rho_N^0) |\log \W_\infty(\rho_N(0),\bar \rho_N^0)|\right) e^{C t}.
	\end{align}
Using the stability result from Theorem \ref{th:intermediate} yields analogously as in the proof of Theorem~\ref{th:tau}
	\begin{align}
			\W_p (\rho_N(t),\rho(t)) \lesssim    \W_\infty (\bar \rho_N(t),\rho_N(t)) + e^{Ct}  \left(\W_p (\rho_N^0,\rho_0) + \W_\infty (\bar \rho_N(0),\rho_N(0))\right).
	\end{align}
	Combining this estimate with the claim above and equation \eqref{eq:W_p.intermediate} and Proposition \ref{pr:rho.rho_eff} yields \eqref{eq:rho_N.rho_eff}.

It remains to prove the claim.
Again, we show that the assumptions of Theorem \ref{th:mean.field.general} are satisfied, this time with 
\begin{align}
	K &= \Phi g , \\
	\varphi_N(x) &= \frac{g}{6 \pi N R}-5 \phi_N  (e \Phi \ast( \tau (e \Phi g \ast \tau)))(x), \\
	e_N(t) &= C (\phi_N^2 |\log \phi_N| + N^{-1/3}) e^{C t}.
\end{align}
Again, by \eqref{eq:discretization.error} and assumption \eqref{eq:density.condition}, we can absorb the  error $N^{-1/3}$ into 
$\W_\infty (\bar \rho_N(0),\rho_N(0)$ to obtain the assertion.

We observe that the regularity of $\tau$ provided by Theorem \ref{th:hoefer} implies that the functions $\varphi_N$ satisfy the assumptions of Theorem \ref{th:mean.field.general}. 

Condition \eqref{eq:Wasserstein.T=0.generic} again follows directly from assumptions \eqref{eq:density.condition}, \eqref{eq:ass.phi.logN} and \eqref{eq:rho.bar.dmin}.

It remains to verify \eqref{eq:E_i.bounds}.
In fact, by Theorem \ref{th:tau}, we already know that  for all $T_\ast > 0$ and for all $N = N(T_\ast)$ sufficiently large
\begin{align} \label{eq:est.dmin}
	\dmin(t) \geq \dmin(0) e^{-C t} \quad \text{for all } t \leq T_\ast.
\end{align}
Therefore, following the proof of Theorem \ref{th:mean.field.general}, we observe that instead of \eqref{eq:E_i.bounds}, it suffices to show that
\begin{align}
		\forall T_\ast > 0 \, \exists N_0> 0 \, \forall N > N_0 \, \forall t \leq T_\ast \quad \sup_i E_i(t) &\leq e_N(t). 
	\end{align}
We fix $T_\ast$ and $t \leq T_\ast$ and assume that $N$ is taken sufficiently large such that \eqref{eq:est.dmin} holds.

In this case the error is given by the $\bar{E}_i$ from \eqref{eq:V_i.strategy2}:
\begin{align} 
	\bar E_i=V_i- \frac{g}{6 \pi N R}  - \sum_{j \neq i} \Phi(X_i - X_j)g + 5 \phi_N (e \Phi \ast (\tau (e \Phi g \ast \tau)))(X_i).
\end{align}
To control it, we apply the results on the method of reflections from \cite{Hoefer19} that we stated in Section \ref{sec:MoR}.  Since by assumption \eqref{eq:density.condition} and \eqref{eq:est.dmin}, 
\begin{align} \label{eq:phi_0}
	c_0(t) = \frac {R^3}{\dmin(t)^3} \lesssim \phi_N e^{C t } \to 0 \quad \text{as } N \to \infty, 
\end{align}  
Theorem \ref{th:reflections} implies for $v_N^{(1)}$ from \eqref{eq:def.v^k}
\begin{align} \label{eq:Cor.2.7}
	\|u_N - v_N^{(1)}\|_{L^\infty} \lesssim (R^{\alpha} + \phi_N^{1/p'}) \phi_N \|e v^{(0)}\|_{L^p(\cup_i B_i)},
\end{align}
where $p> 3$ and $\alpha = 1 - 3/p$ and we used Lemma \ref{le:sums} and \eqref{eq:density.condition} to estimate $\lambda_p\ls \phi_N$ in \eqref{eq:lambda_q}.
We can estimate the right-hand side of \eqref{eq:Cor.2.7} using $v^{(0)} = \sum_i w_N(\cdot - X_i)$ and the explicit form of $w_N$ from \eqref{eq:w_n.explicit}:
\begin{align}
	\|e v^{(0)}\|^p_{L^p(\cup_i B_i)} \lesssim R^3 \sum_{i} \left( \sum_{j \neq i} \frac{1}{N |X_i-X_j|^2} \right)^p \lesssim N R^3 \left(\frac{1}{N^{2/3} {\dmin(t)^2} } \right)^p   \lesssim \phi_N e^{C t},
\end{align}
where we used Lemma \ref{le:sums} and \eqref{eq:density.condition}.
Thus, 
\begin{align} 
	\|u_N - v_N^{(1)}\|_{L^\infty} \lesssim e^{Ct} (R^{\alpha} \phi_N^{-1/p'} + 1 ) \phi_N^2.
\end{align}
We have 
\begin{align} 
	R^{\alpha} \phi_N^{-1/p'} \lesssim R^{1-3/p}(NR^3)^{1/p-1}=R^{-2}N^{1/p-1}\ls \dmin^{-2}N^{1/p-1}\ls N^{1/p-1/3}.
\end{align}
Since $p > 3$,
\begin{align} \label{eq:u_N.v^1}
	\|u_N - v_N^{(1)}\|_{L^\infty} \lesssim  e^{Ct} \phi_N^2.
\end{align}

Thus (since $V_i=u_N(X_i)$), in order to prove the claim, it suffices to prove that for all $i$
\begin{align} \label{eq:V_i.V_i.app}
	|V_{i,\app} - v_N^{(1)}(X_i)| \lesssim \left( N^{-1/3} + \phi_N^2 |\log \phi_N| + \phi_N \W_\infty(\rho_N(0),\rho_0) (1+ |\log \W_\infty(\rho_N(0),\rho_0)|)\right) e^{C t},
\end{align}
where
\begin{align}
	V_{i,\app} := \frac{g}{6 \pi N R}  + \sum_{j \neq i} \Phi(X_i - X_j)g - 5 \phi_N (e \Phi \ast (\tau (e \Phi g \ast \tau)))(X_i).
\end{align}

We use {\cite[Lemma 3.1]{Hoefer19}}, implying that for $x \in B_i$
\begin{align}
	Q_i v_N^{(0)} (x) = v_N^{(0)}(x) - \fint_{\partial B_i} v_N^{(0)} \dd y - \frac 1 2 \fint_{B_i} \curl v_N^{(0)} \dd y  \times (x - X_i).
\end{align}
In particular, for all $x \in B_i$
\begin{align} \label{eq:Q_i.B_i}
	|Q_i v_N^{(0)} (x)| \leq  R \|\nabla v_N^{(0)}\|_{L^\infty(B_i)} \lesssim R \frac 1 N \sum_{j \neq i} |X_i - X_j|^2 \lesssim  R e^{Ct},
\end{align}
where we used Lemma \ref{le:sums} as well as \eqref{eq:density.condition}.

Thus, by the definition of $v^{(1)}$ and the explicit form of $w_N$, we have
\begin{align}
	v_N^{(1)}(X_i) =& \frac{g}{6 \pi N R} + \sum_{j \neq i}  w_N(X_i - X_j)  - Q_i v_N^{(0)} (X_i) -  \sum_{j \neq i} \sum_{k \neq j} (Q_j w_N(\cdot - X_k))(X_i) \\
	=& V_{i,\app} + \sum_{j \neq i} \left( w_N(X_i - X_j)  - \frac 1 N \Phi(X_i - X_j) g\right) - Q_i v_N^{(0)} (X_i) \\
	&- \sum_{j \neq i} \sum_{k \neq j} (Q_j w_N(\cdot - X_k))(X_i) + 5 \phi_N (e \Phi \ast( \tau (e \Phi g \ast \tau)))(X_i) \\
	=:& V_{i,\app} + E_{i,1} - E_{i,2}.
\end{align}
It remains to estimate $E_{i,1}$ and $E_{i,2}$.
As regards $E_{i,1}$, we use
\eqref{eq:Q_i.B_i}, \eqref{eq:w_n.explicit} and Lemma \ref{le:sums} as well as \eqref{eq:density.condition} to obtain
\begin{align} \label{eq:est.E_i,1}
	|E_{i,1}| &\lesssim 
	R e^{Ct} + \frac {R^2} N \sum_{j \neq i} |X_i - X_j|^{-3} \lesssim e^{Ct} (R+ R \alpha_2)  \lesssim e^{Ct} R \lesssim  e^{Ct} N^{-1/3}.
\end{align}

The second error term $E_{i,2}$, we further split into
\begin{align}
	E_{i,2} =& \sum_{j \neq i} \sum_{k \neq j} \left( (Q_j w_N(\cdot - X_k))(X_i) - \frac{5 \phi_N}{N^2} e \Phi(X_i - X_j) (e \Phi(X_j - X_k) g) \right) \\
	&+ \sum_{j \neq i} \sum_{k \neq j} \frac {5 \phi_N} {N^2} e \Phi(X_i - X_j)
	 (e \Phi(X_j - X_k) g - 5 \phi_N (e \Phi \ast( \tau (e \Phi g \ast \tau)))(X_i) \\
	=:& E_{i,2,1} + E_{i,2,2}.
\end{align}
We first estimate $E_{i,2,1}$. By Lemma \ref{lem:dipole},
\begin{equation} \label{eq:Q_j.D}
\begin{aligned} 
	& \left|(Q_j w_N(\cdot - X_k))(X_i) - \frac{5 \phi_N}{N} e \Phi (X_i - X_j) \fint_{B_j} e w_N(x - X_k) g \dd x\right| \\
	&\lesssim \frac{R^{5/2}}{|X_i - X_j|^3} \|e w_N(\cdot - X_k)\|_{L^2(B_j)}  \lesssim \frac{R^4}{N} \frac{1}{|X_i -X_j|^3} \frac 1 {|X_j - X_k|^2}.
\end{aligned}
\end{equation}
where the last inequality above follows directly from  \eqref{eq:w_n.explicit}.

With \eqref{eq:Q_j.D} and 
\begin{align}
	\fint_{B_j} |\frac 1 N e \Phi(X_j - X_k) g  -  e w_N (x - X_k) | \dd x \leq \frac{R}{N} |X_j - X_k|^{-3},
\end{align}
we find
\begin{align}
	|E_{i,2,1}| &\lesssim \sum_{j \neq i} \sum_{k \neq j} \left(\frac{R^4}{N} \frac{1}{|X_i -X_j|^3} \frac 1 {|X_j - X_k|^2} +
	\frac{R^4}{N} \frac{1}{|X_i -X_j|^2} \frac 1 {|X_j - X_k|^3} \right) \\
	& \lesssim e^{Ct} R \phi_N \log{N} \lesssim e^{Ct} R \lesssim  e^{Ct} N^{-1/3}, 
\end{align}
where we used Lemma \ref{le:sums} and \eqref{eq:ass.phi.logN} in the last line.

It remains to estimate $E_{i,2,2}$ 
We use the existence of an optimal transport plan $T$ such that $\rho_N = T \# \tau$ and
\begin{align}
	\W_\infty(\rho_N,\tau) = \tau - \esssup |T (x) - x|.
\end{align}   
With the convention $\Phi(0) =0$, $e \Phi (0) = 0$, this leads to 
\begin{align*}
	E_{i,2,2} =& 5 \phi_N \iint e \Phi(X_i - x) e \Phi(x-y) g \dd \rho_N(x) \dd \rho_N(y) \\
	& - 5 \phi_N \iint e \Phi(X_i - x) e \Phi(x-y) g  \tau(x)  \tau(y) \dd x  \dd y \\
	=& 5 \phi_N \iint \Bigl(e \Phi(X_i - T (x)) e \Phi(T(x)-T(y)) g  - e \Phi(X_i - x) e \Phi(x-y) g  \Bigr) \tau(x) \tau(y) \dd x   \dd y \\
	=& 5 \phi_N \int  \left(e \Phi(X_i - T (x))  - e \Phi(X_i -x) \right) \int e \Phi(T(x)-T(y)) g  \tau(y) \dd y   \tau(x) \dd x \\
	&+ 5 \phi_N \int e \Phi(X_i - x) \int \left(e \Phi(T(x)-T(y)) g - e \Phi(x-y) g  \right)  \tau(y)  \dd y \tau(x) \dd x   \\
	=:& 5 \phi_N(I_1 + I_2).
\end{align*}
Regarding $I_1$, we first bound the inner integral uniformly in $x$  using that, for $\tau$-a.e. $x$, $T x = X_j$  for some $1 \leq j \leq N$. Thus,
\begin{align}
	\left|\int e \Phi(T(x)-T(y)) g  \tau(y) \dd y \right| &=  \left| \frac 1 N \sum_{k \neq j}  e \Phi(X_j - X_k) g  \right| 
	\lesssim \frac 1 N \sum_{k \neq j} \frac{1}{|X_k - X_j|^2} \lesssim e^{Ct}
\end{align}
due to Lemma \ref{le:sums}.
Thus, using Lemma \ref{lem:sums.Wasserstein} similarly as in the proof of Theorem \ref{th:mean.field.general},
\begin{align}
	|I_1| \lesssim & \int_{\R^3 \setminus B_{2 \W_\infty(\rho_N,\tau)}(X_i)} |T (x) - x| \left(
	\frac 1 {|X_i  -T(x)|^{3}} + \frac 1 {|X_i   - x|^3} \right) \tau(x) \dd x  \\
	& + 
	\int_{B_{2 \W_\infty(\rho_N,\tau)}(X_i)} \left(\frac 1 {|X_i  -T(x)|^2} + \frac 1 {|X_i   - x|^2} \right) \tau(x)   \dd x \\
	 \lesssim &e^{Ct}\W_\infty(\rho_N,\tau) \int_{\R^3 \setminus B_{2 \W_\infty(\rho_N,\tau)}(X_i)}  \frac 1 {|X_i   - x|^3}  \tau(x)  \dd x  
	+ e^{Ct}\W_\infty(\rho_N,\tau) \\
	\lesssim & e^{Ct} \W_\infty(\rho_N,\tau)(1+|\log \W_\infty(\rho_N,\tau)|).
\end{align}
Regarding $I_2$, we obtain in a very similar way
\begin{align}
	|I_2| \lesssim  e^{Ct}\W_\infty(\rho_N,\tau)(1+|\log \W_\infty(\rho_N,\tau)|).
\end{align}
Indeed, since $\|e \Phi \ast h\| \ls \| h\|_{L^1 \cap L^\infty}$, it suffices to get an $L^\infty$-bound  on the inner integral.
This is obtained by splitting the integral just as above.

To conclude the proof of the claim, we apply the result of Theorem \ref{th:tau} to bound $\W_\infty(\rho_N,\tau)$.
This yields, using monotonicity of $z(1+ |\log z|)$ and $\phi_N \ll 1$,
\begin{align}
	|E_{i,2,2}| \lesssim e^{Ct} \left(\phi_N^2 |\log \phi_N| + \phi_N \W_\infty(\rho_N(0),\rho_0) (1+ | \log \W_\infty(\rho_N(0),\rho_0)|) \right).
\end{align}
Combining these error estimates yields the desired estimate \eqref{eq:V_i.V_i.app}

\medskip 

	 It remains to prove \eqref{eq:main.u_eff-u_N}.
	 Again, by the stability result from Theorem \ref{th:intermediate} and Proposition \ref{pr:rho.rho_eff}, it suffices to estimate 
	 $\|\bar u_N - u_N\|_{L^q_\loc}$ where $\bar u_N$ is the fluid velocity in the solution to \eqref{eq:u.rho} with initial data $\bar \rho_N^0$. 
	 Then, 
	 \begin{align}
	 	\bar u_N = \Phi g \ast \bar \rho_N - 5 \phi_N e \Phi \ast (\tau (e \Phi g \ast \tau)).
	 \end{align} 
	 Thus,
	 \begin{align} \label{eq:bar.u_n.u_N}
	 	|\bar u_N - u_N| \leq |v_N^{(1)} - u_N| + | v^{(1)} - \Phi g  \ast \rho_N + 5 \phi_N e \Phi \ast (\tau (e \Phi g \ast \tau))| + |\Phi g \ast (\bar \rho_N - \rho_N)|. \qquad
	 \end{align}
	 Again, the first term on the right-hand side, we control by \eqref{eq:u_N.v^1},
	 and the third term by \eqref{eq:estimate.convolutions}. The second term is estimated similarly as above, namely, we claim
	 \begin{equation}  \label{eq:v^1.convolution}
	 \begin{aligned}
	 	&| v^{(1)} - \Phi g  \ast \rho_N + 5 \phi_N e \Phi \ast (\tau (e \Phi g \ast \tau))|(x) \lesssim \frac 1 {N |\dist(x,\{X_i\}_i)|} \\
	 	& + e^{Ct} \left(N^{-1/3}+ \phi_N^2 |\log \phi_N| + \phi_N \W_\infty(\rho_N(0),\rho_0) (1+| \log \W_\infty(\rho_N(0),\rho_0)|) \right).
	 \end{aligned}
	 \end{equation}
		To see this, recall from the proof of Theorem \ref{th:tau} (equation \eqref{v^0.convolution}) that 
	 \begin{align} 
	 	| v^{(0)} - \Phi g \ast \rho_N|(x) \lesssim \frac 1 {N |\dist(x,\{X_i\}_i)|} + R e^{Ct}.
	 \end{align}
	 Thus, it suffices to show
	 \begin{align}
	 	&\left| \sum_i Q_i v^{(0)} - 5 \phi_N e \Phi \ast (\tau (e \Phi g \ast \tau))\right|(x) \\ 
	 	& \lesssim   e^{Ct} \left(N^{-1/3} + \phi_N^2 |\log \phi_N| + \phi_N \W_\infty(\rho_N(0),\rho_0) (1+ | \log \W_\infty(\rho_N(0),\rho_0)|) \right).
	 \end{align}
	 The left-hand side is the error term $E_{i,2}$ above except that we are evaluating in $x$ instead of $X_i$ and the sum also runs over $i$. Therefore, the previous inequality follows along the same lines as above, provided we show
	 \begin{align} \label{eq:Q_i.pointwise.everywhere}
	 	 |(Q_i v^{(0)})(x)| \lesssim R \lesssim N^{-1/3}.
	 \end{align}
	 To see this, we use that  by \eqref{eq:Q_i.B_i}, \eqref{eq:Q_i.pointwise.everywhere} also holds in $B_i$. Thus, by the maximum modulus theorem for the Stokes equations (see e.g. \cite{MaremontiRussoStarita99}), \eqref{eq:Q_i.pointwise.everywhere} holds for all $x \in \R^3$. Thus \eqref{eq:v^1.convolution} holds. 
	 
	 Inserting this estimate in \eqref{eq:bar.u_n.u_N} together with \eqref{eq:u_N.v^1} and \eqref{eq:estimate.convolutions} yields
	 	 \begin{align}
	 	&|\bar v_N - u_N|(x)  \lesssim \frac 1 {N |\dist(x,\{X_i\}_i)|}  \\
	 	& + e^{Ct} \left(N^{-1/3}+ \phi_N^2 |\log \phi_N| + \phi_N \W_\infty(\rho_N(0),\rho_0) | \log \W_\infty(\rho_N(0),\rho_0)|  + \W_\infty(\rho_N(0),\bar \rho_N^0\right) ) e^{Ct}. 
	 \end{align}
	 Using $\alpha_2 \lesssim e^{Ct}$, $R + N^{-1/3} + \W_\infty(\rho_N(0),\bar \rho_N^0)\lesssim  \W_p(\rho_N(0), \rho_0)$ and \eqref{eq:int.dist.X_i}, yields for all $q < 3$
	 \begin{align}
	 	\|\bar v_N - u_N\|_{L^q_{\loc}} \lesssim & e^{Ct} \left(N^{-1} \dmin(0)^{-1} + R^3 \dmin(0)^{-2}  + \phi_N^2 |\log \phi_N|  \right.\\
	 	& \left. + \phi_N \W_\infty(\rho_N(0),\rho_0) | \log \W_\infty(\rho_N(0),\rho_0)|  + \W_p(\rho_N(0), \rho_0)\right) ) e^{Ct} \\
	 	 \lesssim & \left(\phi_N^2 |\log \phi_N| + \phi_N \W_\infty(\rho_N(0),\rho_0) | \log \W_\infty(\rho_N(0),\rho_0)|  + \W_p(\rho_N(0), \rho_0)\right) ) e^{Ct}.
	 \end{align}
	 This finishes the proof.
\end{proof}

\section{Well-posedness and estimates for the macroscopic systems}
\label{sec:continuous}

In this section, we prove the results on the continuous systems \eqref{eq:v.tau}, \eqref{eq:u.rho} 
and \eqref{eq:rho_eff.u_eff} stated in Section \ref{sec:transition}.
We begin with the proof of Proposition \ref{pr:rho.rho_eff}, which is the remaining ingredient in the proof of Theorem \ref{th:main}. As before, we will rely on the well-posedness and estimates for the continuous systems which we will prove in Section \ref{sec:well-posedness}. 

\subsection{Proof of Proposition \ref{pr:rho.rho_eff}}

We will use the following slight generalization of \cite[Proposition 2.8]{Loeper06}:
\begin{prop}\label{prop:Loeper}
 Let $\nu_0,\nu_1\in \fP(\R^d)\cap L^\infty(\R^d)$. Then for all $p\in [1,\infty)$ it holds:
 \begin{align}
   \norm{\nu_0-\nu_1}_{W^{-1,p}(\R^d)}\le \max(\norm{\nu_0}_{L^\infty},\norm{\nu_1}_{L^\infty})^{1/{p^\prime}}\W_p(\nu_0,\nu_1).
 \end{align}
\end{prop}
\begin{rem}
 In \cite[Exercise 38]{Santambrogio15}, one can find an even more general statement.
\end{rem}
For self-containedness we give the outline of the proof here. 
\begin{proof}
 Since $\nu_0,\nu_1$ are absolutely continuous w.r.t. the Lebesgue measure, there exists an (a.e. unique) optimal transport map $T$ s.t.
  \begin{align}
  \W_p(\nu_0,\nu_1)=\int \abs{x-T(x)}^p\nu_0(x)\dd x.
 \end{align}
  Consider for $\theta\in (0,1)$:
 \begin{align}
  \nu_\theta=(\theta T +(1-\theta) \Id)\#\nu_0.
 \end{align}
 By \cite[Theorem 5.27]{Santambrogio15} $\nu_\theta$ is a constant-speed geodesic connecting $\nu_0,\nu_1$. Therefore, we have by \cite[Proposition 7.29]{Santambrogio15}
 \begin{align}
  \norm{\nu_\theta}_{L^\infty}\le \max\set{\norm{\nu_0}_{L^\infty},\norm{\nu_1}_{L^\infty}}.
 \end{align}
 By definition, for any $\nu\in W^{-1,p}$,
  \begin{align}
  \norm{\nu}_{W^{-1,p}}=\sup\set{\int \nu \varphi: \varphi\in C_0^\infty(\R^d), \norm{\nabla \varphi}_{L^{p^\prime}}\le 1},
 \end{align}
 where $\frac 1 p +\frac{1}{p^\prime}=1$.
 Take any $\varphi\in C^\infty_0(\R^d)$. Then
 \begin{align}
  \int \nu_\theta(x)\varphi(x)\dd x=\int \nu_0(x) \varphi(\theta T(x)+(1-\theta)x)\dd x.
 \end{align}
 Differentiating this, we obtain
 \begin{align}
  \frac{\dd}{\dd \theta}\int \nu_\theta(x)\varphi(x)\dd x=\int \nu_0(x) \nabla \varphi(\theta T(x)+(1-\theta)x)(T(x)-x)\dd x.
 \end{align}
 Using the Cauchy-Schwartz inequality, we arrive at
 \begin{align}
  \frac{\dd}{\dd \theta}\int \nu_\theta(x)\varphi(x)\dd x&\le \bra{\int \nu_0(x)\abs{T(x)-x}^p\dd x}^{1/p}\bra{\int \nu_\theta \abs{\nabla \varphi(x)}^{p^\prime}}^{1/{p^\prime}}\\
  &\le \W_p(\nu_1,\nu_2)\norm{\nu_\theta}_{L^\infty}^{1/{p^\prime}}\norm{\nabla\varphi}_{L^{p^\prime}}.
 \end{align}
 Integration with respect to $\theta$ yields the statement. 
\end{proof}
\begin{proof}[Proof of Proposition \ref{pr:rho.rho_eff}]
We start by proving that
\begin{align}\label{eq:rho_eff.tau}
  \W_\infty(\rho_\eff(t),\tau(t))\ls C(T_\ast,\norm{\rho_0}_{W^{1,1}\cap W^{1,\infty}})\phi_N.
\end{align}
Let $f_N$ be defined as in  Lemma \ref{lem:Wass.deriv}, i.e., 
 \begin{align}
  f_N(t)=\sup_{s \in [0,t]} \tau(s) - \esssup\abs{T_s(x)-x}.
\end{align}
where $T_t=Y(t,0,\cdot)\circ X(0,t,\cdot)$ and  $X,Y$ are the flow maps corresponding to $\tau$, $u_\eff$, respectively.
Then, by Lemma \ref{lem:Wass.deriv} we have for all $0 \leq t_1 \leq t_2$
\begin{align}\label{eq:x_t}
   f_N(t_2) - f_N(t_1) &\le \int_{t_1}^{t_2} \norm{u_\eff(s)\circ T_s-v(s)}_\infty \dd s \\
   &\le\int_{t_1}^{t_2} \norm{v(s)-u_\eff(s)}_{L^\infty}+\norm{\nabla v(s)}_{L^\infty} f_N(s) \dd s.
\end{align}
In order to estimate the difference of $v$ and $u_\eff$ we write down the difference of \eqref{eq:v.tau} and \eqref{eq:rho_eff.u_eff} in the following form:
\begin{align}\label{eq:u_eff.v}
  -\Delta(u_\eff-v)+\nabla p=(\rho_\eff-\tau)g+\dv\bra{5\phi_N\rho_\eff eu_\eff}, && \dv (u_\eff-v) = 0 .
\end{align}
Thus, $u_\eff-v$ is given by the convolution of the right-hand side with the Oseen tensor:
\begin{align*}
 &(u_\eff-v)(s,x)=\int_{\R^3}\Phi(x-y)\bra{\rho_\eff(y)-\tau(y)+\dv\bra{5\phi_N\rho_\eff eu_\eff}(y)}\dd y\\
 &=\int_{\R^3}(\Phi(x-T_s(y))-\Phi(x-y))\tau(y)\dd y-\int_{\R^3}\nabla \Phi(x-y)\bra{5\phi_N\rho_\eff(y) eu_\eff(y)}\dd y.
\end{align*}
Using \eqref{eq:Phi.Lipschitz} (and $\norm{\tau}_{L^1}=\norm{\rho_\eff}_{L^1}=1$) as well as Theorem \ref{th:hoefer} and Theorem \ref{thm:effective} we estimate 
\begin{align}\label{eq:u_eff.v.est}
 \norm{(u_\eff-v)}_{L^\infty}\ls& \int_{\R^3}\abs{T(y)-y}\left(\frac{1}{\abs{x-T(y)}^2}+\frac{1}{\abs{x-y}^2}\right)\tau(y)\dd y\\
 &+\phi_N\int_{\R^3}\frac{1}{\abs{x-y}^2}\rho_\eff(y) eu_\eff(y)\dd y\\
 \ls&  f_N \bra{1+\norm{\tau}_{L^\infty}+\norm{\rho_\eff}_{L^\infty}}+\phi_N(1+\norm{\rho_\eff}_{L^\infty})\norm{\nabla u_\eff}_{L^\infty}\dd y\\
 \le& C(T_\ast,\norm{\rho_0}_{W^{1,1}\cap W^{1,\infty}})(f_N+\phi_N).
\end{align}
Inserting \eqref{eq:u_eff.v.est} into \eqref{eq:x_t} and using Theorem \ref{th:hoefer} again, yields
\begin{align}
 f_N(t_2) - f_N(t_1) &\le C(T_\ast,\norm{\rho_0}_{W^{1,1}\cap W^{1,\infty}}) \int_{t_1}^{t_2} f_N +\phi_N\dd s.
\end{align}
%
In view of Gronwall's inequality and \eqref{eq:f.larger.Wass} this entails \eqref{eq:rho_eff.tau}.

\medskip

We now turn to the proof of the assertion. Notice, that it is enough to prove \eqref{eq:prop_continuous} for large $p$ since we have $\W_q\le \W_p$ for $q<p$ (see \eqref{eq:W.hirarchy}).
Let 
 \begin{align}
  h_N(t)=\sup_{s\le t}\norm{(S_s-\Id)\rho(s)^{1/p}}_{L^p},
\end{align}
where $S_t=Y(t,0,\cdot)\circ Z(0,t,\cdot)$ and  $Y,Z$ are the flow maps corresponding to $\rho_\eff$, and $\rho$, respectively.

Then, Lemma \ref{lem:Wass.deriv} yields for all $0 \leq t_1 \leq t_2$
\begin{align}\label{eq:h1}
h_N(t_2) - h_N(t_1) \leq \int_{t_1}^{t_2} \norm{(u_\eff(s)-u(s))}_{L^p}\norm{\rho_\eff}^{1/p}_{L^\infty}+\norm{\nabla u(s)}_{L^\infty}h_N(s) \dd s.
\end{align}
It remains to estimate $\norm{(u_\eff-u)}_{L^p}$. To this end we subtract \eqref{eq:rho_eff.u_eff} from \eqref{eq:u.rho} to get
\begin{align}\label{eq:u.u_eff}
  -\Delta(u-u_\eff)+\nabla p&=(\rho-\rho_\eff)g+\dv\bra{5\phi_N(\tau ev-\rho_\eff eu_\eff)}\\
  &=(\rho-\rho_\eff)g+5\phi_N\dv\bra{(\tau-\rho_\eff)ev+\rho_\eff(ev-eu_\eff)}.
\end{align}
This equation is the reason why we cannot extend the result to $p=\infty$. Due to the divergence term on the right-hand side, it seems difficult to obtain a good $L^\infty$ estimate for the difference of $u$ and $u_\eff$.

 By linearity we treat each of the terms on the right-hand side separately. We consider the solutions to 
 \begin{align}
 -\Delta w_1+\nabla p=(\rho-\rho_\eff)g, \quad \dv w_1 = 0,\\
 -\Delta w_2+\nabla p=\dv\bra{(\tau-\rho_\eff)ev}, \quad \dv w_2 = 0,\\
 -\Delta w_3+\nabla p=\dv\bra{\rho_\eff(ev-eu_\eff)}, \quad \dv w_3 = 0. \label{eq:w_3}
\end{align}
Let $p>3/2$ and $\frac 1 q=\frac 1 p +\frac 1 3$. Then we have, using Proposition \ref{prop:Loeper},  \eqref{eq:W.hirarchy} and \eqref{eq:f.larger.Wass}: 
\begin{align}\label{eq:w_1}
 \norm{w_1}_{L^p}&\ls \norm{\nabla w_1}_{L^q}\ls \norm{\rho-\rho_\eff}_{W^{-1,q}}\ls \W_q(\rho,\rho_\eff)(\norm{\rho}_{L^\infty}+\norm{\rho_\eff}_{L^\infty})^{1/q'}\\
 &\ls C(\norm{\rho_0}_{L^\infty})h_N.
\end{align}
We now turn to $w_2$. Let $r>3$. Using Proposition \ref{prop:Loeper} and Theorem \ref{th:hoefer}: 
\begin{align}\label{eq:w_2}
 \norm{w_2}_{L^p}&\ls\norm{(\tau-\rho_\eff)ev}_{W^{-1,p}}\\
 &\ls \norm{(\tau-\rho_\eff)}_{W^{-1,p}}\norm{ev}_{W^{1,r}}\\
 &\ls \W_p(\tau,\rho_\eff)(\norm{\tau}_{L^\infty}+\norm{\rho_\eff}_{L^\infty})^{1/p'}(\norm{\tau}_{L^\infty})\\
 &\le C(\norm{\rho_0}_{L^\infty})\phi_N,
\end{align}
where we used, that for $\varphi\in W^{1,p'},\psi\in W^{1,r}$ the product $\varphi\psi$ is again in $W^{1,p'}$ whence 
\begin{align}\label{eq:tau.rho_eff}
 \norm{(\tau-\rho_\eff)ev}_{W^{-1,p}}\ls \norm{(\tau-\rho_\eff)}_{W^{-1,p}}\norm{ev}_{W^{1,r}}.
\end{align}
We now turn to $w_3$. First, equation \eqref{eq:u_eff.v} implies that we have for every $q\in (1,\infty)$ the bound 
\begin{align}
 \norm{\nabla(u_\eff-v)}_{L^q}&\ls \norm{\rho_\eff-\tau}_{W^{-1,q}}+\phi_N\norm{\dv(\rho_\eff eu_\eff)}_{W^{-1,q}}\\
 &\ls \W_q(\rho_\eff,\tau)(\norm{\rho_\eff}_{L^\infty}+\norm{\tau}_{L^\infty})^{1/q'}+\phi_N\norm{\rho_\eff eu_\eff}_{L^q}\\
 &\ls C(\norm{\rho_0}_{L^\infty})\W_q(\rho_\eff,\tau)+\phi_N\norm{\rho_\eff}_{L^\infty} \norm{\nabla u_\eff}_{L^q}\\
 &\ls C(\norm{\rho_0}_{W^{1,1}\cap W^{1,\infty}})(\W_\infty(\rho_\eff,\tau)+\phi_N)\\
 &\ls C(\norm{\rho_0}_{W^{1,1}\cap W^{1,\infty}})\phi_N,
\end{align}
where we used Theorem \ref{thm:effective} and \eqref{eq:rho_eff.tau}. Equation \eqref{eq:w_3} thus yields that for $p>3/2$ and $\frac 1 q=\frac 1 p+\frac 1 3$:
\begin{align}\label{eq:w_3_est}
\begin{aligned}
  \norm{w_3}_{L^p}&\ls \norm{\nabla w_3}_{L^q}\ls \norm{\dv\bra{\rho_\eff(ev-eu_\eff)}}_{W^{-1,q}}=\norm{\rho_\eff(ev-eu_\eff)}_{L^q}\\
  &\ls \norm{\rho_\eff}_{L^\infty}\norm{\nabla(u_\eff-v)}_{L^q}\ls C(\norm{\rho_0}_{W^{1,1}\cap W^{1,\infty}})\phi_N.
\end{aligned}
\end{align}
We notice that $u-u_\eff=w_1+5\phi_N(w_2+w_3)$. Combining \eqref{eq:w_1}, \eqref{eq:w_2} and \eqref{eq:w_3_est} yields
\begin{align}
\norm{(u_\eff-u)(s)}_{L^p}\ls \W_p(\rho,\rho_\eff)+C(\norm{\rho_0}_{W^{1,1}\cap W^{1,\infty}})\phi_N^2.
\end{align}
Inserting this into \eqref{eq:h1}, we obtain
\begin{align}
h_N(t_2) - h_N(t_1) \leq C(\norm{\rho_0}_{W^{1,1}\cap W^{1,\infty}}) \int_{t_1}^{t_2} (h_N+\phi_N^2).
\end{align}
By Gronwall's inequality and \eqref{eq:f.larger.Wass}, this yields the assertion.
\end{proof}

\subsection{Proof of the well-posedness theorems} \label{sec:well-posedness}
 We will here only provide  in full detail the proof of Theorem \ref{thm:effective}, the well-posedness of the effective system \eqref{eq:rho_eff.u_eff}.
The proofs of Theorems \ref{th:hoefer} and \ref{th:intermediate} are easier.
Indeed, system \eqref{eq:rho_eff.u_eff} has the most complicated structure of the three systems \eqref{eq:v.tau},
\eqref{eq:u.rho} and \eqref{eq:rho_eff.u_eff} since it contains an additional nonlinearity in the coefficients of the Stokes equations. In turn, the stability statement is easier to prove in Theorem \ref{th:hoefer} than in Theorem \ref{th:intermediate}, which is why we only give the proof for the latter. Since the coefficients in system \eqref{eq:v.tau} are constant, one needs to impose less regularity for the initial data to get the same differentiability of the velocity field. 

We will start with the proof of Theorem \ref{thm:effective} which is already, in parts, similar to the proof of Proposition \ref{pr:rho.rho_eff}. Afterwards, we will provide the proof of the stability statement in Theorem \ref{th:intermediate} where we can recycle arguments from Proposition \ref{pr:rho.rho_eff} and Theorem \ref{thm:effective}. Finally, we will comment shortly on the changes for Theorem \ref{th:hoefer}.

\begin{proof}[Proof of Theorem \ref{thm:effective}]
For the purpose of brevity and because we only consider one type of equation in this proof, we will omit all subscripts $\eff$.  Moreover, we omit the constant velocity $(6 \pi \gamma_N)^{-1}$ since it can be removed from the system by a change of variables.

We will first prove estimates for a decoupled linear problem in order to apply a fixed point argument.
For $3<p<\infty$ consider $\nu \in W^{1,1}\cap W^{1,p}$ (so that in particular $\nu\in L^\infty$) and the equation
\begin{align*}
-\dv \bra{\bra{2+5\phi\nu}eu} + \nabla p&=\nu,\\
-\dv u&=0.
\end{align*}
The right hand side is in $L^1\cap L^\infty$ and hence in every $L^q$ for $q\in [1,\infty]$. Convolution of the above equation with the Oseen-Tensor yields
\begin{align*}
u-\phi \Phi\ast \dv \bra{5\nu eu}&=\Phi\ast\nu.
\end{align*}
Note that the right-hand side is in $W^{2,q}$, while the left-hand side has the structure $(\Id+\phi T)u$ where the operator $T$ is given by 
$Tu=-\Phi\ast \dv \bra{5\nu eu}$. The operator $T$ is a bounded operator from $W^{2,q}\to W^{2,q}$ for $q\in (3,p]$ with norm bounded by $\norm{\nu}_{W^{1,p}}$. Thus, for small enough $\phi$, it is possible to invert $\Id+\phi T$. This implies $u\in W^{2,q}$ for all $q\in (3,p]$ and $u\in W^{1,\infty}$ by embedding with the corresponding estimates in terms of $\nu$:
%
%
%
%
\begin{align}\label{eq:reg.est}
\norm{u}_{W^{1,\infty}}\le C\norm{\nu}_{W^{1,1}\cap W^{1,p}}.
\end{align} 
For given $\nu\in L^\infty(0,T;W^{1,1}\cap W^{1,p}\cap \fP)$, consider the following problem for $\rho$:
\begin{align} \label{eq:noncoupled_eff}
\left\{\begin{array}{rcl}
-\dv((2+5\phi \nu)eu)+\nabla p&=&\nu g,\\
\dv u&=&0,\\
\partial_t \rho+u\cdot \nabla \rho&=&0, \\
\rho(0) &=& \rho_0.
\end{array}\right.
\end{align}
Since $u\in L^\infty(0,T;W^{1,\infty})$, the solution to \eqref{eq:noncoupled_eff} exists and we have $\norm{\rho(t)}_{L^\infty}=\norm{\rho_0}_{L^\infty}$ as well as $\norm{\rho(t)}_{L^1}=\norm{\rho_0}_{L^1}$.

The gradient of $\rho$ satisfies the following equation:
\begin{align} \label{eq:gradient_rho}
\partial_t \nabla \rho+(u\cdot \nabla) \nabla\rho+(\nabla u)^T \nabla\rho&=0, \\
\nabla \rho(0) &= \nabla \rho_0.
\end{align}
The coefficients $u,\nabla u$ are regular enough ($u\in L^\infty(0,T;W^{1,\infty})$), so that DiPerna-Lions theory \cite{DiPernaLions89} yields a unique solution $\nabla\rho \in L^\infty(0,T;L^q)$ to this problem for all $q\in [1,p]$. 
Thus $\nabla \rho\in L^\infty(0,T;L^1\cap L^p)$ with the following estimate:
\begin{align}\label{eq:grad.growth}
\begin{aligned}
\norm{\nabla\rho}_{L^\infty(0,T;L^q)}&\le\norm{\nabla \rho_0}_{L^q}\exp\bra{\int_0^{T}\norm{\nabla u(t)}_{L^\infty}\dd t} \\
&\le \norm{\nabla \rho_0}_{L^q}\exp\bra{T\norm{\nu}_{L^\infty(0,T;W^{1,1}\cap W^{1,p})}}, 
\end{aligned}
\end{align}
where we used \eqref{eq:reg.est}. Thus, for small enough $T$, the map $A$ that maps $\nu$ to $\rho$, i.e. $\rho=A(\nu)$ maps the set 
\begin{align} 
B_1=\set{\nu:\norm{\nu}_{L^\infty(W^{1,p}\cap W^{1,1})}\le C_1},
\end{align} 
to itself, where $C_1$ has to be chosen larger than $\norm{\rho_0}_{W^{1,p}\cap W^{1,1}}$. Let $X$ be the flow map corresponding to \eqref{eq:noncoupled_eff} and $T_t(\cdot)=X(t,0,\cdot)$. Then $T_t$ is a transport plan with $\rho(t)=T_t\#\rho_0$. We compute using \eqref{eq:reg.est}: 
\begin{align}\label{eq:dist.rho0}
\W_p(\rho(t),\rho_0)&\le\bra{\int_{\R^3}\abs{T_t(x)-x}^p\rho_0(x)\dd x}^{1/p}\\
&\leq \int_0^t\bra{\int_{\R^3}\abs{u(s,T_s(x))}^p\rho_0(x)\dd x}^{1/p}\dd s\\
&\le \int_0^t\bra{\int_{\R^3}\abs{u(s,x)}^p\rho(s,x)\dd x}^{1/p}\dd s\\
&\le t\norm{u}_{L^\infty(0,T;L^\infty)}\\
&\ls T\norm{\nu}_{L^\infty(0,T;W^{1,p}\cap W^{1,1})}.
\end{align} 
Therefore $A$ maps 
\begin{align} 
B=\set{\nu:\W_p(\nu(s),\rho_0)<\infty \text{ for all }s\in [0,T], \norm{\nu}_{L^\infty(0,T;W^{1,p}\cap W^{1,1})}\le C_1},
\end{align} 
to itself. For the foregoing argument we could have chosen the (simpler) $\W_\infty$ distance but we are only able to prove contractivity in $\W_p$. We want to prove contractivity of $A$ in $B$ w.r.t. $L^\infty(\W_p)$.
$B$ is complete with respect to this metric, since in view of the bound on the $L^\infty(0,T;W^{1,p}\cap W^{1,1})$ norm in $B$ any Cauchy sequence has a subsequence that converges weakly-$\ast$ in $L^\infty(0,T;W^{1,p})$ and its limit coincides with the limit in the Wasserstein metric (and thus the whole sequence converges weakly-$\ast$ in $L^\infty(0,T;W^{1,p})$) and the limit satisfies the same bound. This convergence also implies that the limit satisfies the $L^\infty(0,T;W^{1,1})$-bound since for almost every time $t \in (0,T)$, the sequence converges weakly in $W^{1,1}_\loc$.

 Let $\nu_1,\nu_2 \in B$. Let $u_1,\rho_1,u_2,\rho_2$ be the corresponding solutions to \eqref{eq:noncoupled_eff}. 
Consider
 \begin{align}
  f(t)=\sup_{s\le t}\norm{(T_s-\Id)\rho_2(s)^{1/p}}_{L^p},
\end{align}
where $T_t=Y_1(t,0,\cdot)\circ Y_2(0,t,\cdot)$ and  $Y_1,Y_2$ are the flow maps corresponding to $\rho_1$, $\rho_2$, respectively. 
 Using Lemma \ref{lem:Wass.deriv} we have for all $0 \leq t_1 \leq t_2$
 \begin{align}\label{eq:exist.wass}
   f(t_2) - f(t_1)   &\le\int_{t_1}^{t_2} \norm{(u_2-u_1)(s)}_{L^p}\norm{\rho_2(s)}^{1/p}_{L^\infty}+\norm{\nabla u_1(s)}_{L^\infty}f(s)\dd s.
\end{align}

By \eqref{eq:reg.est} we have $\norm{\nabla u_1}_{L^\infty(0,T;L^\infty)}\ls \norm{\nu_1}_{L^\infty(0,T;W^{1,p}\cap W^{1,1})}\le C_1$. It remains to estimate the difference of $u_1$ and $u_2$. We argue similarly to the proof of Proposition \ref{pr:rho.rho_eff} following Equation \eqref{eq:u.u_eff}. We start by writing 
\begin{align}
  -\dv((2+5\phi\nu_1)(eu_1-eu_2))+\nabla p&=(\nu_1-\nu_2)g+5\phi\dv\bra{(\nu_1-\nu_2) eu_2)}.
\end{align}
The left-hand side is the sum of the solutions $w_1,w_2$ to
\begin{align}
  -\dv((2+5\phi\nu_1)ew_1)+\nabla p&=(\nu_1-\nu_2)g,\\
  -\dv((2+5\phi\nu_1)ew_2)+\nabla p&=5\phi\dv\bra{(\nu_1-\nu_2) eu_2)}.
\end{align}
Let $\frac 1 q =\frac 1 p +\frac 1 3 $. Building on the optimal regularity of the Stokes equations, by a perturbative argument as the one preceding \eqref{eq:reg.est}, we know that
\begin{align}
\norm{\nabla w_1}_{L^q}\ls \norm{\nu_1-\nu_2}_{W^{-1,q}},
\end{align}
and, recalling \eqref{eq:tau.rho_eff}:
\begin{align}
\norm{w_2}_{L^p}&\ls \norm{(\nu_1-\nu_2)eu_2}_{W^{-1,p}}\\
&\ls  \norm{(\nu_1-\nu_2)}_{W^{-1,p}}\norm{eu_2}_{W^{1,p}}.
\end{align}
Combining these inequalities and using Proposition \ref{prop:Loeper} and yields
\begin{align}
\norm{u_1-u_2}_{L^p}&=\norm{w_1+w_2}_{L^p}\\
&\ls \norm{\nabla w_1}_{L^q}+\norm{w_2}_{L^p}\\
&\ls \norm{\nu_1-\nu_2}_{W^{-1,q}}+\norm{(\nu_1-\nu_2)}_{W^{-1,p}}\norm{eu_2}_{W^{1,p}}\\
&\ls \bra{\W_q(\nu_1,\nu_2)+\W_p(\nu_1,\nu_2)}\bra{1+\norm{\nu_1}_{L^\infty}+\norm{\nu_2}_{L^\infty}}\norm{\nu_2}_{W^{1,p}\cap W^{1,1}}\\
&\ls (1+C_1^2) \W_p(\nu_1,\nu_2).
\end{align}
If we insert this into equation \eqref{eq:exist.wass} we obtain
\begin{align}
f(t_2) - f(t_1)\ls C \int_{t_1}^{t_2} \bra{\W_p(\nu_1(s),\nu_2(s))+f(s)} \dd s,
\end{align}
which implies by a Gronwall type argument and \eqref{eq:f.larger.Wass}
\begin{align}
\W_p(\rho_1,\rho_2)&\le  \sup_{s\in [0,T]}\W_p(\nu_1(s),\nu_2(s))(e^{Ct}-1).
\end{align}
This implies that $A$ is a contraction for small $T$ and thus admits a fixed point $\rho_\eff$ in $B$ which is a solution to \eqref{eq:rho_eff.u_eff}. 

By \eqref{eq:dist.rho0} and  \eqref{eq:grad.growth} both the $\W_p$-distance to $\rho_0$ and the $L^\infty(W^{1,1}\cap W^{1,p})$ norm stay bounded for any $T$. Thus the solution must exist for all times. 
Notice that by the method of characteristics \eqref{eq:grad.growth} holds in particular for $q=\infty$. Thus $\rho\in W^{1,\infty}$ and the norm bound is the same as for $p<\infty$ by taking the limit $p\to \infty$.
\end{proof}

\begin{proof}[Proof of Theorem \ref{th:intermediate}]
The proof of existence is largely analogous to the one for Theorem \ref{thm:effective}. The Einstein term on the left-hand side can be regarded as a right-hand side term since it is already known. The achievable regularity, however, stays the same as for the effective model.

The stability result is proved in the same way as \cite[Theorem 3.1]{Hauray09}.
For completeness, we give provide the short proof in the case $1 \leq p < \infty$. The case $p = \infty$ is analogous. 
Consider
 \begin{align}
  f(t)=\sup_{s\le t}\norm{(T_s-\Id)\rho_2(s)^{1/p}}_{L^p},
\end{align}
where $T_t=Y_1(t,0,\cdot)\circ Y_2(0,t,\cdot)$ and  $Y_1,Y_2$ are the flow maps corresponding to $\rho_1$, $\rho_2$, respectively. 
By Lemma \ref{lem:Wass.deriv} we have for all $0 \leq t_1 \leq t_2$
\begin{align}\label{eq:derivest}
   f(t_2) - f(t_1)   &\le\int_{t_1}^{t_2} \norm{(u_2-u_1)(s)}_{L^p}\norm{\rho_2(s)}^{1/p}_{L^\infty}+\norm{\nabla u_1(s)}_{L^\infty}f(s)\dd s.
\end{align}
We notice that we have
\begin{align}
-\Delta(u_1-u_2)+\nabla p=(\rho_1-\rho_2)g,
\end{align} 
which yields, using Jensen's inequality
\begin{align}
|u_2 (Tx) -u_1(x)|^p
&= \left(\int \Phi(Tx - y)g \rho_2(y) - \Phi(x - y)\rho_1(y) \dd y \right)^p  \\
& \lesssim \left(\int |\Phi(Tx - Ty) - \Phi(x - y)| \rho_1(y) \dd y \right)^p \\
& \lesssim  \left(\int (|x-Tx| + |y - Ty|) \left(\frac{1}{|Tx - Ty|^2} + \frac 1 {|x-y|^2} \right) \rho_1(y) \dd y \right)^p  \\
& \lesssim  \int (|x-Tx|^p + |y - Ty|^p) \left(\frac{1}{|Tx - Ty|^2} + \frac 1 {|x-y|^2} \right) \rho_1(y) \dd y \\
&\quad \times \left( \int \left(\frac{1}{|Tx - Ty|^2} + \frac 1 {|x-y|^2} \right) \rho_1(y) \dd y   \right)^{p-1}.
\end{align}
Thus, 
\begin{align}
 &\norm{(u_2 \circ T -u_1)\rho_1^{1/p}}_{L^p} \\
&\lesssim \left( \iint (|x-Tx|^p + |y - Ty|^p) \left(\frac{1}{|Tx - Ty|^2} + \frac 1 {|x-y|^2} \right) \rho_1(y) \rho_1(x) \dd y \dd x \right)^{\frac 1 p} \\
& \quad\times\sup_x \left( \int \left(\frac{1}{|Tx - Ty|^2} + \frac 1 {|x-y|^2} \right) \rho_1(y) \dd y   \right)^{\frac{p-1} p} \\
&\lesssim  \W_p (\rho_1,\rho_2) \sup_x \int \left(\frac{1}{|Tx - Ty|^2} + \frac 1 {|x-y|^2} \right) \rho_1(y) \dd y \\
&\lesssim (1 +  \|\rho_1\|_{L^\infty} + \|\rho_2\|_{L^\infty}) \W_p (\rho_1,\rho_2).
\end{align}

Inserting the above inequality into \eqref{eq:derivest} and employing a Gronwall argument for $\W_p(\rho_1,\rho_2)$ yields the stability in $\W_p$ in view of \eqref{eq:f.larger.Wass}.

To estimate $u_1-u_2$, let $p> 3/2$ and $1/q = 1/ p + 1/3$.
Then we have, using Proposition \ref{prop:Loeper}: 
\begin{align}
 \norm{u_1-u_2}_{L^p}&\ls \norm{\nabla (u_1-u_2)}_{L^q}\ls \norm{\rho_1-\rho_2}_{W^{-1,q}}\ls \W_q(\rho_1,\rho_2)(\norm{\rho_1}_{L^\infty}+\norm{\rho_2}_{L^\infty})^{1/q'}\\
 &\ls C\W_q(\rho_0^1,\rho_0^2).
\end{align}
Using \eqref{eq:W.hirarchy} finishes the proof.
\end{proof}

\begin{proof}[Proof of Theorem \ref{th:hoefer}]
The existence proof is again analogous to the one for Theorem \ref{thm:effective}. Since there is no Einstein term on the left-hand side, the velocity field $v$ possesses two weak derivative more than the right-hand side and we do not need derivatives of the density. We do not need to consider the transport equation for the gradient. The rest of the argument is basically unchanged. See also \cite{Mecherbet20}. The proof of the stability estimate is completely analogous to the one in the proof of Theorem \ref{th:intermediate} since by subtracting the equations for two solutions the Einstein term is annihilated.
\end{proof}

\section*{Acknowledgements}

We warmly thank Juan J.L. Vel\'azquez for discussions. The first author has been supported by the Deutsche Forschungsgemeinschaft (DFG, German Research Foundation) through the collaborative research center ``The Mathematics of Emerging Effects'' (CRC 1060,Projekt-ID 211504053) and the Hausdorff Center for Mathematics (GZ 2047/1, Projekt-ID 390685813). The second author was associated member of the Research Training Group (DFG Graduierten Kolleg) funded by the DFG – Projektnummer 320021702/GRK2326 – ``Energy, Entropy, and Dissipative Dynamics (EDDy)'' at the RWTH Aachen while conducting the research for this article and thanks the Institute for Applied Mathematics Bonn for its hospitality during the research stays in Bonn.
\appendix
\section{Wasserstein distances}\label{app:Wasserstein}

In this appendix we briefly review the most important definitions regarding Wasserstein distances.

\begin{defi}
For two probability measures $\mu,\nu \in \P(\R^3)$ the set of couplings $\Gamma(\mu,\nu)$ is defined as the set of all probability measures $\gamma\in \P(\R^3\times \R^3)$ with first marginal $\mu$ and second marginal $\nu$, i.e.
\begin{align}
\gamma\in \Gamma(\mu,\nu) \iff \int_{\R^3\times \R^3}(\varphi(x)+\psi(y))\dd \gamma(x,y)=\int_{\R^3}\varphi(x)\dd \mu(x)+\int_{\R^3}\psi(x)\dd \nu(x)
\end{align}
for all bounded and continuous functions $\varphi,\psi: \R^3\to \R$.
\end{defi}

\begin{defi}
Let $\mu,\nu \in \P(\R^3)$. Then, for $p\in [1,\infty)$, the $p$-Wasserstein distance is defined as
\begin{align}
\W_p(\mu,\nu)=\inf_{\gamma\in \Gamma(\mu,\nu)}\bra{\int_{\R^3\times \R^3}\abs{x-y}^p\dd \gamma(x,y)}^{1/p}.
\end{align}
For $p=\infty$ we set
\begin{align}
\W_\infty(\mu,\nu)=\inf_{\gamma\in \Gamma(\mu,\nu)}\gamma-\esssup\abs{x-y}.
\end{align}
\end{defi}

In several important cases (and in this article), because of additional information on the probability measures, the existence of an optimal transport map $T$ is ensured such that $\nu=T\#\mu$ and such that $(\Id,T\# \mu)$ is a minimizer for the Wasserstein distance. In this case $T$ is also optimal with respect to all other maps $T'$ with $\nu=T'\#\mu$.

\printbibliography

\end{document}